\documentclass[twoside, 12pt, reqno]{amsart}

\DeclareSymbolFont{symbols}{OMS}{cmsy}{m}{n}

\usepackage{fullpage}
\usepackage[utf8]{inputenc}
\usepackage[T1]{fontenc}

\usepackage{mathrsfs}
\usepackage{mathtools}
\mathtoolsset{showonlyrefs,showmanualtags}
\usepackage{amsfonts}
\usepackage{amsmath, amssymb, amsthm}
\usepackage{enumerate}
\usepackage{graphicx, tikz}
\usepackage{ifthen}
\usetikzlibrary{arrows, scopes}
\usetikzlibrary{calc, cd}
\usetikzlibrary{decorations.pathreplacing}
\usetikzlibrary{decorations.pathmorphing}
\usetikzlibrary{patterns}
\usetikzlibrary{fadings}
\usetikzlibrary{cd}
\usetikzlibrary{automata, positioning}
\usetikzlibrary{calc}
\usepackage{tikz-cd}
\usepackage{stmaryrd}
\usepackage{setspace}
\usepackage{hyperref}
\usepackage[alphabetic, initials]{amsrefs}
\usepackage{bbm}

\newtheorem{thm}{Theorem}[section]
\newtheorem{lem}[thm]{Lemma}
\newtheorem{cor}[thm]{Corollary}
\newtheorem{prop}[thm]{Proposition}
\theoremstyle{definition}
\newtheorem{defi}[thm]{Definition}
\newtheorem{example}[thm]{Example}
\theoremstyle{remark}
\newtheorem{rmk}[thm]{Remark}
\newtheorem{conj}[thm]{Conjecture}

\newcommand{\Hil}{\mathcal{H}}
\newcommand{\slot}{{~\cdot~}}

\newcommand{\SL}{\mathop{\mathsf{SL}}}

\newcommand{\PSL}{\mathop{\mathsf{PSL}}}

\newcommand{\U}{{\mathsf{U}}}

\newcommand{\Mob}{\mathsf{M\ddot ob}}

\newcommand{\W}{\mathcal{W}}

\newcommand{\cC}{\mathcal{C}}

\newcommand{\cB}{\mathcal{B}}

\newcommand{\A}{\mathcal{A}}
\newcommand{\cI}{\mathcal{I}}

\newcommand{\RR}{\mathbb{R}}
\newcommand{\TT}{\mathbb{T}}

\newcommand{\CC}{\mathbb{C}}

\newcommand{\ZZ}{\mathbb{Z}}
\newcommand{\NN}{\mathbb{N}}

\DeclareMathOperator{\Diff}{Diff}
\DeclareMathOperator{\End}{End}

\DeclareMathOperator{\Rep}{Rep}

\DeclareMathOperator{\Hom}{Hom}

\DeclareMathOperator{\diag}{diag}

\DeclareMathOperator{\Ad}{Ad}

\DeclareMathOperator{\im}{Im}
\DeclareMathOperator{\id}{id}

\DeclareMathOperator{\B}{B}
\DeclareMathOperator{\supp}{supp}
\DeclareMathOperator{\sign}{sign}

\DeclareMathOperator{\Aut}{Aut}
\newcommand{\punkt}{\,\mathrm{.}}
\newcommand{\komma}{\,\mathrm{,}}

\newcommand{\e}{\mathrm{e}}

\newcommand{\ima}{\mathrm{i}}

\renewcommand{\Im}{\im}

\usepackage{xspace}
\DeclareRobustCommand{\eg}{e.g.\@\xspace}

\DeclareRobustCommand{\ie}{i.e.\@\xspace}
\makeatletter
\DeclareRobustCommand{\etc}{%
    \@ifnextchar{.}%
        {etc}%
        {etc.\@\xspace}%
}
\makeatother

\def\u1net{{\A_\RR}}

\renewcommand{\H}{\Hil}

\DeclareMathOperator{\Out}{\mathsf{Out}}
\DeclareMathOperator{\Inn}{\mathsf{Inn}}

\renewcommand{\L}{\mathrm{L}}
\newcommand{\cG}{\mathcal{G}}

\newcommand{\utr}{\triangle}

\newcommand{\SOne}{\mathrm{S}^1}
\newcommand{\ide}{\mathrm{id}}
\newcommand{\ev}{\mathrm{ev}}
\newcommand{\coev}{\mathrm{coev}}

\DeclareMathOperator{\Ima}{Im}

\newcommand{\Borel}{\mathrm{Borel}}
\newcommand{\tRep}[1]{#1\text{--}\Rep}

\begin{document}
\date{\today}
\dateposted{\today}
\newcommand{\mytitle}{%
Anomalies for conformal nets associated with lattices and $T$-kernels }
\title{\mytitle}
\author{Marcel Bischoff}
\address{Independent Scholar, Columbus, OH 43209, USA} \email{marcel@localconformal.net}

\thanks{Supported by NSF DMS grant 1700192/1821162}
\author{Pradyut Karmakar}
\address{Department of Mathematics, Morton Hall 423, 1 Ohio University, Athens, OH 45701, USA}
\email{pk481519@ohio.edu} \dedicatory{This article is dedicated to the memory of Vaughan Fredrick Randall Jones}

\begin{abstract}
  Let $L\subseteq \RR^{n}$ an even lattice and $T_{L}=\RR^{n}/L$ the associated
  torus. Associated with $L$ we construct $T_{L}$--kernel on a hyperfinite factor
  type $\A_{L}$, \ie a monomorphism $T_{L}\to\Out(\A_{L})$, and compute Sutherland's
  \cite{Su1980} obstruction class in $H^{3}_{\Borel}(T_{L},\TT)\cong H^{4}(BT_{L}
  ,\ZZ)$, which is an invariant of the $T_{L}$--kernel and an obstruction to the
  existence of a twisted crossed product by $T_{L}$. As a Corollary, we obtain
  that for any $n$-torus $T$ any class in $H^{3}_{\mathrm{Borel}}(T,\TT)$ arises
  as an obstruction for a $T$-kernel on the hyperfinite type III${}_{1}$
  factor $R$.

  The construction is an analogue of the construction of Vaughan Jones for
  finite groups on the hyperfinite type II${}_{1}$ factor but is also motivated
  by and has applications to conformal nets. Namely, there is an associated local
  extension $\A_{L}\supseteq \A_{\RR^n}$ of conformal nets and the $T_{L}$--kernel
  corresponds to a family of $T_{L^\ast}$--twisted sectors representations whose
  anomaly (obstruction) can be identified with the inner product on $L$ seen as
  a class in $H^{4}(BT_{L},\ZZ)\cong \operatorname{Sym}^{2}(L,\ZZ)$.
\end{abstract}

\maketitle

\setcounter{tocdepth}{3}
\tableofcontents

\section{Introduction}
\label{sec:Introduction} 
Anomalies in quantum field theory have profound insights into the underlying symmetries and structures of fundamental physics.
 In theoretical physics, anomalies arise when a classical symmetry of a system fails to be preserved at the quantum level due to quantum defects. 
The break-down of the symmetry influence the behaviour of particles and fields in unexpected ways.

Group extensions play a crucial role in understanding anomalies by providing a mathematical framework to investigate the symmetries of quantum field theories. 
A group extension is a way of constructing a larger group from two given groups, typically involving a quotient group and a subgroup. 
In the context of anomalies, group extensions help elucidate how symmetries behave under certain transformations and how anomalies manifest as deviations from these expected symmetries.

In this paper, we study anomalies arising from compact groups $\SOne$ and $n$-dimensional torus for any $n \in \NN$.
Anomalies have a classical analogy with group extensions.
Let $N$ be a normal subgroup of $G$.
Then there is a canonical homomorphism $\pi_1 \colon G/N \to \Out(N)$ given by
\begin{align*}
G/N \ni [gN] \mapsto [\alpha_g] \in \Out(N), \text{where } \alpha_g(x)=gxg^{-1}~\text{for all}~ x \in N\, \text{ and for all }g \in G\,.
\end{align*}
If $[gN]=[hN]$, then $g^{-1}h \in N$, therefore $\Ad(g^{-1}h) \in \Inn(N)$.
Thus $[\alpha_{g^{-1}h}]=0$, which implies that $[\alpha_g]=[\alpha_h]$.
Thus $\pi_1$ is well-defined homomorphism. 
Note that if $H^{3}(G/N, Z(N))$ is trivial then group extensions always exist and all the extensions are identified with the group $H^{2}(G/N, Z(N))$.
However, if $H^{3}(G/N, Z(N))$ is non-trivial, then it is an obstruction to make the extension.

Sutherland generalized group extensions to von Neumann algebra extensions.
A von Neumann algebra $M$ is called an extension by a von Neumann algebra of $N$ by a locally compact group $G$ if $M$ is generated by a isomorphic copy of $N$ ($\iota(N)$) and a group of Borel unitaries ($\{\pi(g): g \in G\}$) normalizing $\iota(N)$ such that $G/\iota(N)$ forms a representation (see \cite{Su1980}) of $M$.
The class $[\omega] \in H^{3}_{\Borel}(G, \U(1))$ is the obstruction to make such extensions. 

The main goal of this paper is to construct anomaly (obstructions) using Loop group models.
Loop group models are conformal nets arising from representation of Loop groups.
In particular, we look into free bosons and nets associated with lattices, \ie $\U(1)$-current net.

Moreover, anomalies appear in pointed Fusion categories.
The class $[\omega] \in H^{3}(G, k^{\ast})$ is incorporated to the associator of the category where $k^{\ast}$ denotes multiplicative group of ground field $k$ and $G$ is a finite group.
For example, for a finite group $G$, take $G$--graded finite dimensional vector spaces over a field $k$, \ie objects are given by $V=\bigoplus_{g \in G} V_g$ and tensor product structure of two vector spaces $V$ and $W$ is given by 
\begin{align*}
(V \otimes W)_g:=\bigoplus_{hk=g} V_h \otimes W_k\,.
\end{align*}
Then the associator twisted by a $3$-cocycle $\omega$ is given by
\begin{align*}
&\alpha_{U, V, W}=\omega(g, h, k) \gamma\\
& \gamma((u \otimes v) \otimes w)=(u \otimes (v \otimes w)), u \in U_g, v \in V_h, w \in W_k\,. 
\end{align*}
We get category $\mathrm{Vec}_{G}^{\omega}$.
Furthermore, any pointed fusion category is isomorphic to $\mathrm{Vec}_{G}^{\omega}$ for some $3$-cocycle $\omega$ and some finite group $G$.

Let $M$ be a factor, \ie a von Neumann algebra with
trivial center. 
To an element $[\alpha]\in \Out(M)$ of order $n$, Connes associates
an invariant 
\begin{align*}
k\in\TT_{n}:=\{z\in \CC : z^{n}=1\}
\end{align*}
and showed that every invariant
arises for the hyperfinite type $\mathrm{II}_1$ factor in \cite[Remark~6.8]{Co1977}. 
More
generally, let $G$ be a locally compact group. A $G$--kernel is a monomorphism
$\theta\colon G\to \Out(M)$ that admits Borel lifts.

To a $G$--kernel, Sutherland associates an invariant in $H^{3}_{\Borel}(G,\U(1)
)$, the third $\Borel$ cohomology group of $G$. 
The invariant is called the obstruction because it is an obstruction to
the existence of a twisted crossed product of $M$ by $G$ implementing the $G$--kernel.
Sutherland's obstruction generalizes Connes in the following sense. 
For $[\alpha
]\in \Out(M)$ of order $n$, the map $m\mapsto[\alpha^{m}]$ defines the $\ZZ/n\ZZ$-kernel given as
$m\pmod n \mapsto[\alpha^{m}]$ and the two obstructions are related by a natural
isomorphism $H^{3}_{\Borel}(\ZZ/n\ZZ,\U(1)) \xrightarrow{\scriptstyle\sim}
\TT_{n}$.

Most importantly, $G$--kernels have a categorical viewpoint.
A $G$--kernel $\theta\colon G\to \Out(M)$ gives a strict 2-group $\cG$ with
objects
\begin{align*}
\cG:=\{\alpha\in\Aut(M) : \alpha\in [\theta_{g}] \text{ for some }g\in G\}\,,
\end{align*}
 morphisms
\begin{align*}
\cG(\alpha,\beta) = \{u\in \U(M): u\alpha(x)=\beta(x)u \text{ for all }x\in M\} \qquad \alpha, \beta \in \Aut(M)
\end{align*}
and monoidal structure $\alpha\otimes\beta=\alpha\circ\beta$. Thus
$\pi_{0}(\cG)\cong G$ and if $M$ is a factor, then $\pi_{1}(\cG)=Z(M)\cong \U(1)$
equipped with the trivial $G$--action and the element in
$H^{3}_{\Borel}(G,\U(1))$ classifies the isomorphism class of $\cG$, but the technicality occurs as we are using Borel cohomology.

The problem of realizing anomaly in operator algebras is not new.
For finite cyclic groups, in \cite {Co1977} Connes showed every obstruction arises in the hyperfinite $\mathrm{II}_1$ factor.
For $G$ a finite group, in \cite{Jo1980} Jones showed that every obstruction arises for the
hyperfinite type $\mathrm{II}_1$ factor in his pioneering work on classification of $\mathrm{II}_1$ factors. 
Later, Sutherland showed that in the hyperfinite $\mathrm{II}_1$ factor anomalies can be realized for discrete groups in \cite{Su1980}.
One can use Izumi's result based on Popa's
classification theory for subfactors (see \cite{Po1994-2}) to show that the same is true for the
hyperfinite type $\mathrm{III}_1$ factor.

Not much seems to be known about continuous groups. 
We have used the approach of  Doplicher--Haag--Roberts (also known as DHR) and generalization of Fredenhagen--Rehren--Schroer to compute the braiding of the sectors associated to ambient nets.
 Longo--Rehren in \cite{LoRe1995} have developed the method of $\alpha$-induction of sectors which is used to extend automorphisms on the crossed-product using our braiding.
  We have adapted the construction of $\A_L$ constructed by Dong--Xu in \cite{DoXu2006} and also in \cite{Bi2011}.

Let $L$ be a positive even lattice and $F:=L\otimes_{\ZZ}
\RR$ its associated Euclidean space.

We construct a ``braided'' outer $F$-action on a hyperfinite type $\mathrm{III}_1$
factor $\A_{F}$ such that for any even lattice $L\subseteq F$ a
certain twisted crossed product by $L$ gives ``local'' extension
$\A_{L}\supseteq \A_{F}$. We show that $\alpha$-induction gives $T_{L}:=F/L$--kernel,
\ie a monomorphism from the torus $T_{L}$ to the group $\Out(\A_{L})$.

In the context of conformal nets, $G$--kernels of finite groups have been studied
and an analogue of Jones' result, namely that any obstructions for any finite group 
arises, was established by Evans and Gannon \cite{EvGa2022}.
In this context, given a \emph{holomorphic} conformal net $\A$, \ie a conformal net 
that does not have any non-trivial local endomorphisms, and a finite group 
$G\subset \Aut(\A)$ the $G$--twisted automorphisms form a $G$--kernel
\cite[4.8 Remark 2]{Mg2005}.
These automorphisms can be constructed using $\alpha$-induction on the fixed point net 
$\A^G$ as shown by M"uger \cite[3.8 Theorem]{Mg2005}.
Furthermore, the representation category of $\A^G$ is the twisted 
unitary Drinfel'd double of $G$ where the twist is exactly the anomaly of the $G$--kernel 
\cite[Theorem 5.14]{Mg2005}.

We consider a similar situation.
Let $L$ be an even positive lattice of rank $n$ inside an Euclidean space $F$.
There is a conformal net $\A_L$ associated with $L$.
Dong and Xu showed \cite{DoXu2006} that the representation category of $\A_L$ has fusion rules 
given by the finite discriminant group $L^\ast/L$, where $L^\ast$ is the dual lattice.
Thus we have a $L^\ast/L$--kernel of \emph{local automorphisms}. 
One can costruct $\A_L$ as the twisted crossed product of the conformal net $\A_F$
associated with the space $F$ by the group $L$. 
This implies that there is a natural action of the $n$-torus $F/L^\ast$ on $\A_L$
such that $\A_F$ is the fixed point under this action.

We construct an $F/L$--kernel from $\alpha$-induction which is an extension
of the symmetry group $F/L^\ast$ by the local automorphisms $L^\ast/L$ group

$$
    {0} \rightarrow L^\ast\to L \rightarrow F/L \rightarrow F/L^{\ast} \rightarrow \{0\}
$$
and explicitly compute its anomaly and give a closed formula.
We show that it corresponds to an element in $H^4(B(F/L), \ZZ)$ 
coming from the pairing of its Gram matrix with the cup product of the standard
generators of $H^2(B(F/L), \ZZ)$.
In particular, this implies our main theorem (Corollary \ref{cor:MainResult}):
\begin{prop}
\label{finalrslt}
    Let $n\in\NN$ and $T_n = \RR^n/\ZZ^n$. 
    Then all possible obstructions for $T_n$--kernels are realized 
    on the hyperfinite type III${}_1$ factor.
\end{prop}

\subsection*{Organization of the paper}
\label{ssec:Organization of the paper} In Section 2, we briefly discuss about the
basics of a $G$--kernel. Subsequently, we construct $\U(1)$-kernel and
$\U(1)^{n}$--kernel by using the the braidings in Proposition~
\ref{prop:BraidingFormula}. In section 4, we compute the anomalies associated
to $\U(1)$--kernel and $\U(1)^{n}$--kernel using conformal field theory. In Theorem~
\ref{thm:CocycleOneDimensional} , we compute Borel cohomology groups
$H^{3}_{\text{Borel}}(\U(1), \U(1))$. 
We compute $H^{3}_{\text{Borel}}(\U(1)^{n}
, \U(1))$ in the subsection \ref{ssec:ComputationOfobstructionNDimensional} . 
In Section 6, we talk about twisted sectors and show that the $F/L$--kernel
is $F/L^{\ast}$--twisted for any even lattice $L$ in an $n$-dimensional
euclidean space $F$. In Section 7, we discuss about extension of endomorphisms
in appendix.
\subsection*{Acknowledgements}
\label{ssec:Acknowledgements} M.\ B.\ would like to thank Alexei Davydov, Luca
Giorgetti, and Corey Jones for their discussions.
The second author would like to thank Vladimir Uspenski for discussions on $H^{1}_{\Borel}(\U(1), \U(1))$.

\section{Preliminaries}
\label{sec:Preliminaries} 

\subsection{\texorpdfstring{$G$--kernel}{G--kernel}}
\label{ssec:GKernel} 
In this section we recall some basics of the notion of a $G$-kernel on a von Neumann
algera $M$. We refer interested readers to see \cite{Su1980} for more details. 
Our main focus is on type $\mathrm{III}$ factors. Recall that a factor is a
von Neumann algebra $M$ with trivial center $Z(M) = M' \cap M = \CC\cdot 1$.

Throughout this paper, we denote by $G$ a \emph{locally compact group} which
is assumed to be \emph{second countable}, \emph{Hausdorff} and by $M$ a \emph{von
Neumann algebra} acting on a \emph{separable Hilbert space} $\Hil$.

We denote by $\Aut(M)$ the group of $\ast$-preserving unital and normal automorphisms
of $M$. We denote by $\U(M)$ the unitary group of the von Neumann algebra $M$.
Furthermore, we denote by $\Out(M)$ the group of unitary equivalence classes
of automorphisms of $M$ whose elements are of the form
$[\alpha]=\{\Ad v \circ \alpha\colon v \in \U(M) \}$ for some
$\alpha \in \Aut(M)$. We denote by $M_{\ast}$ the predual of $M$, which we
assume to be separable. The topology on $\Aut(M)$ is defined by pointwise norm
convergence in $M_{\ast}$. A sequence $\{\alpha_{n}\}_{n \in \NN}$ in
$\Aut(M)$ is said to converge the element $\alpha$ in $\Aut(M)$ in $u$-toplogy
if and only if for each $x$ in $M$ and for all $\phi$ in $M_{\ast}$, we have
\begin{align*}
  \lim_{n \to \infty}\phi(\alpha_{n}(x))= \phi(\alpha(x))\,.
\end{align*}
The automorphism group $\Aut(M)$ is a Polish topological group under this ambient
$u$-topology. In other words, we call a map $\alpha \colon G \rightarrow \Aut(M
)$ is Borel measurable if for every open set $V$ in $\Aut(M)$ the preimage $\alpha
^{-1}(V)$ is a Borel set.
Recall the definition of a $G$--kernel, see
\cite[Definition~3.1.1]{Su1980} and references therein.
\begin{defi}
  Given a group $G$ and a von Neumann algebra $M$, a monomorphism
  $\theta\colon G \rightarrow \Out(M)$ is called a $G$--kernel if there exists
  a Borel lift $\alpha\colon G \rightarrow \Aut(M)$ such that $\epsilon \circ \alpha
  =\theta$, where $\epsilon\colon \Aut(M) \rightarrow \Out(M)$ is given by
  $\epsilon(\alpha) = [\alpha]$.
\end{defi}

Before going to the next section, we record some basic facts about the Borel cohomology
on a von Neumann algebra and its obstruction theory, both of which are
essential for our purpose. A systematic study of Borel cohomology on von Neumann
algebras has been carried out by Sutherland, Connes, see (\eg
\cite{Co1977, Su1980}). Denote by $\{M, \theta\}$ the $G$--kernel  associated to a von Neumann algebra $M$.
\subsection{ Borel cohomology}
\label{ssec:Borel cohomology}
\begin{defi}
  \label{cocyle}

  Let $A$ be an additive group and $G$ a multiplicative group. Let $\beta \colon
  G \rightarrow \Aut(A)$ be a given Borel map such that $\beta_{g}$ is a Borel
  homeomorphism on $A$. Then we denote $C^{n}(G, A)$ as the set of Borel maps
  from $G^{n}$ to $A$, namely known as Borel $n$-cochains. 
  The boundary map
  $\partial_{n} \colon C^{n}(G, A) \rightarrow C^{n+1}(G, A)$ (for all $n \in \NN$) is defined by
  \begin{align*}
    \partial_{n}\omega(g_{1},\cdots,g_{n+1}) & =\beta_{g_1}(\omega(g_{2},\cdots,g_{n+1}))+                                                                         \\
                                             & \sum_{j=1}^{n}(-1)^{j}\omega(g_{1},\cdots, g_{j} g_{j+1}, \cdots , g_{n+1})+ (-1)^{n+1}\omega(g_{1}, \cdots, g_{n})
  \end{align*}
  for all $\omega \in C^{n}(G, A)$.
\end{defi}
It is clear that $\partial_{n+1}\circ \partial_{n}=0$ for all
$n \in \mathbb{N}$. When $M$ is a factor, \ie, when $M$ has a trivial center,
observe that $\U(M)=\U(1)$, where $\U(1)$ denotes the unit circle group. In
this article, we will work with the case $A=\U(1)$. Note that
$\partial_{n}^{2}=0$ for all $n$. Denote the $n$-cocyles by $Z_{\beta}^{n}(G, A
):=\ker(\partial_{n})$ and $n$-coboundaries by $B_{\beta}^{n}(G, A):=\text{Image}
(\partial_{n-1})$ for all $n \in \NN$.

\begin{defi}
  \label{Cohmgroup} The $n$-th cohomology group of the given $\beta$, denoted
  by $H_{\beta}^{n}(G, A)$, is defined as
  \begin{align*}
    H^{n}_{\beta}(G, A)=Z_{\beta}^{n}(G, A)\slash B_{\beta}^{n}(G, A)\,.
  \end{align*}
\end{defi}
In what follows, we recall the definition of the obstruction associated with
the given $G$--kernel.
\begin{defi}
  Let $\{M, \theta\}$ be a given $G$--kernel and $\alpha\colon G\to\Aut(M)$ a
  Borel lift of $\theta$. By definition, we have
  $\theta_{g} \theta_{h}=\theta_{gh}$ for all $g,h\in G$. Therefore there exist
  unitaries $u(g, h) \in \U(M)$ for any $g, h$ in $G$, such that
  $\alpha\colon G \rightarrow \Aut(M)$ satisfies
  $\alpha_{g} \alpha_{h}=u(g, h) \alpha_{gh}$ for all $g, h \in G$. Comparing the
  equations $(\alpha_{g} \alpha_{h}) \alpha_{k}$ and $\alpha_{g} (\alpha_{h} \alpha
  _{k})$, we have
  \begin{align}
    \label{obt1}\Ad (\alpha_{g}(u(h, k))u(g,hk)) & =\Ad (u(g,h)u(gh,k))\,.
  \end{align}
\end{defi}

Here $\Ad u$ means conjugation by the unitary $u$. Therefore by Equation~\eqref{obt1},
we get
\begin{align}
  \label{obt2}\alpha_{g}(u(h,k))u(g,hk) & =\omega(g, h, k)u(g,h)u(gh,k)\,.
\end{align}
Thus the class $[\omega]$ of $\omega\colon G^{3}\rightarrow \CC$ is called the
\emph{obstruction} associated with the $G$--kernel $\{M, \theta\}$.

Connes and Jones showed the existence of a $G$-kernel with prescribed
obstruction in $\mathrm{II}_{1}$ factor setting. In the same vain, the key
question we address is the following: given a cohomological obstruction, can
one construct examples of $G$--kernel for some locally compact group $G$ in hyperfinite
type $\mathrm{III}_{1}$ factor setting such that the \emph{obstruction} is
realized? The question is more interesting when $G=\U(1)$, $G=\U(1)^{n}$. The answer
is affirmative and triggered by the Conformal field theory.
\subsection{Braiding on tensor subcategory of endomorphisms} 
\label{ssec:Braiding of tensor subcategory of endomorphisms}
Denote by $\End(M)$ the set of unital normal endomorphisms of a type $\mathrm{III}$ factor $M$.
The tensor structure $\otimes$ on endomorphisms (objects) of $\rho$ and $\sigma$ is given by composition, \ie $\rho \otimes \sigma:=\rho \circ \sigma$ and morphisms on objects are given by intertwiners, \ie $\Hom(\rho\sigma, \sigma \rho):=(\rho \sigma, \sigma \rho)=\{ t \in M: t \rho(x)=\sigma(x)t \text{ for all } x \in M\}$.
The tensor structure on the morphisms is given by
\begin{align*}
r \otimes s=r \rho(s)=\tilde{\rho}(s)r, r \in (\rho, \tilde{\rho}), s \in (\sigma, \tilde{\sigma})\,.
\end{align*}
An endomorphism $\rho$ is called irreducible if $\rho(M)' \cap M=\CC 1_{M}$.
A tensor category $\cC$ is called rigid if every element has dual object. More precisely, for each $c \in \cC$ there exists $\hat{c}$ satisfying the following:
\begin{align}
&(\ide_c \otimes \ev_c) \circ (\coev_c \otimes \ide_c)= \ide_c\\
&(ev_c \otimes \ide_{\hat{c}}) \circ (\ide_{\hat{c}} \otimes \coev_c)=\ide_{\hat{c}}\,,
\end{align}
where evaluation map and coevaluation map are given by 
$\ev_c \colon \hat{c} \to 1_{\cC}$ and $\coev_c \colon 1_{\cC} \to c \otimes \hat{c}$, respectively,
and for each element $c \in \cC$ has predual $c_{\ast}$ such that $\hat{c_{\ast}} \cong c$.
A unitary braiding on a rigid $C^{\ast}$-tensor category (closed under finite direct sums and each object is finite sum of irreducible objects) $\cC \subset \End(M)$ is family of unitaries 
in $\{\varepsilon_{\rho, \sigma} \in (\rho\sigma, \sigma \rho): \rho, \sigma \in \cC\}$ satisfying the following:
\begin{itemize}
\item naturality: $\varepsilon_{\rho', \sigma'}s\rho(t) = t\sigma(s)\varepsilon_{\rho, \sigma}$ for all $s \in (\rho, \rho'), t \in (\sigma, \sigma')$,
\item hexagonal diagram equations:
\begin{align*}
&\varepsilon_{\rho_1, \rho_2\rho_3}=(1_{\rho_2} \otimes \varepsilon_{\rho_1, \rho_3}) \cdot (\varepsilon_{\rho_1, \rho_2} \otimes 1_{\rho_3})=\rho_2(\varepsilon_{\rho_1, \rho_3}) \cdot \varepsilon(\rho_1,\rho_2)\,.\\
& \varepsilon_{\rho_1\rho_2, \rho_3}=(\varepsilon_{\rho_1\rho_3} \otimes 1_{\rho_2}) \cdot (1_{\rho_1} \otimes \varepsilon _{\rho_2, \rho_3})=\varepsilon_{\rho_1, \rho_3} \cdot \rho_1(\varepsilon_{\rho_2, \rho_3}) \text{ for all } \rho_1, \rho_2, \rho_3 \in \cC\,.
\end{align*}
\end{itemize}
\subsection{Conformal nets}
\label{ssec:conformalnets}
Let $\cI$ be the set of non-empty, non-dense, open intervals (also known as proper) of the unit circle $S^1$ and $\PSL_2(\RR)=\SL_2(\RR)/\{-1, 1\}$ with $\PSL_2(\RR)$ acts on $S^1$ by Möbius transformations.
Denote as $I'=S^1 \setminus \bar{I}$.
We denote by $\mathcal{B}(\Hil)$ the set of bounded operators on a Hilbert space $\Hil$ and $\U(\Hil)$ the unitary group.

\begin{defi}\label{mnett}\cite[Definition~2.1]{BiDeGi2023}.
A Möbius covariant net on $S^1$ is a triple $(\A, U, \Omega)$ consisting of a family of von Neumann algebras $\{\A(I): I \in \cI\}$ acting on a complex separable Hilbert space $\Hil$, a strongly continuous unitary representation $U \colon \PSL_2(\RR) \rightarrow \U(\Hil)$, a unit vector $\Omega \in \Hil$ such that 
\begin{itemize}
\item \textbf{Isotony}: $\A(I_1) \subset \A(I_2)$ if $I_1 \subset I_2$ with $I_1, I_2 \in \cI$.
\item \textbf{Locality}: $\A(I_1) \subset \A(I_2)'$ if $I_1$ and $I_2$ are disjoint intervals in $\cI$. 
Here $\A(I)'$ denotes the commutant of $\A(I)$ in $\mathcal{B}(\Hil)$ for any $I \in \cI$.
\item \textbf{Möbius covariance}: For each $I \in \cI$ and $g \in \PSL_2(\RR)$, we have 
\begin{align*}
U(g)\A(I)U(g)^{\ast}=\A(gI) \,.
\end{align*}
\item \textbf{Positivity of energy}: $U$ has positive energy, \ie conformal Hamiltonian generator of one-parameter rotation subgroup of $\PSL_2(\RR)$ has non-negative spectrum.
\item \textbf{Vacuum vector}: $\Omega$ is the unique vector upto a phase such that $U(g)\Omega=\Omega$ for all $g \in \PSL_2(\RR)$ and 
\begin{align*}
\overline{\bigvee_{I \in \cI} \A(I)\Omega}=\Hil \,,
\end{align*}
where $\bigvee_{I \in \cI} \A(I)$ is the von Neumann algebra generated by $\A(I)$ for all $I \in \cI$.
\end{itemize} 
\end{defi}
Let $\Diff_{+}(S^1)$ be the group of orientation preserving diffeomorphims of $S^1$.
\begin{defi}\label{cft}\cite[Defintion~2.2]{BiDeGi2023}
A Möbius covariant net $\A$ is called a conformal net if it satisfies another extra condition:
\begin{itemize}
\item The representation $U \colon \PSL_2(\RR) \rightarrow \U(\Hil)$ extends to a projective positive energy representation of $\Diff_{+}(S^1)$ (again denoted by $U$) subject to the conditions:
\begin{align*}
& U(g)\A(I) U(g)^{-1}= \A(gI) \qquad g \in \Diff_{+}(S^1),~ I \in \cI\\
& U(g)xU(g)^{-1}=x \qquad x \in \A(I),~ g \in \Diff_{+}(I')\,,
\end{align*}
where 
\begin{align*}
\Diff_{+}(I')=\{ g \in \Diff_{+}(S^1): g(z)=z \text{ for all } z \in I\}\,.
\end{align*}
\end{itemize}
\end{defi}

Any conformal net $\A$ satisfies the property $\A(I)'=\A(I')$ known as \emph{Haag Duality} (see \cite{BrGuLo1993}).
Furthermore, any conformal net $\A$ fulfills following properties:
\begin{itemize}
\item \emph{Reeh–Schlieder theorem}: The vaccum vector $\Omega$ is cyclic and separating for $\A(I)$ for all $I \in \cI$, \ie
$\overline{\A(I)\Omega}=\Hil$ and $\overline{\A(I)'\Omega}=\Hil$. (See for e.g \cite[Corollary~2.8]{GaFr1993})
\item \emph{Factoriality}: For all $I \in \cI$, $\A(I)$ is a type $\mathrm{III}_1$ factor by \emph{Connes classification theory}, namely \cite{Co1973}.

\end{itemize}
\subsection{\emph{DHR} endomorphisms and \emph{DHR} braiding}
\label{ssec:DHR}
We introduce some terminology around \emph{DHR endomorphisms} of an irreducible net of algebras $\A$.
A representation of a net $\A$ is a pair $(\pi, \Hil_\pi)$ where $\Hil_\pi$ is a 
Hilbert space and $\pi$ is a family of representations:
\begin{align*}
    \pi = \{ \pi_I \colon \A(I) \to \B(\Hil_{\pi}) \}_{I \in \cI}
 \end{align*}
 such that $\pi_{\tilde{I}}|\A(I)=\pi_{I}, I \subset \tilde{I}$.
A representation $\rho$ is called localized on an interval $I_0$ if $\Hil_{\pi}=\Hil$ and $\rho_{I'}=\id_{\A(I')}$.
Let $\rho$ be a localized representation on $I_0$ then Haag duality ($\A(I)'=\A(I')$) implies $\rho_I \in \End(\A(I))$ for any proper interval $I \supset I_0$. 
In this case, $\rho$ is called a \emph{DHR endomorphism} localized in $I_0$.

Let $\Rep^{I}(\A) \subset \End(\A(I_0))$ be the category of \emph{DHR endomorphisms} on $I_0$.
This is a full and replete subcategory of $\End(\A(I_0))$ and has an obvious tensor structure (tensor category).
By abuse of notation, we write $\rho:=\rho_{I_0}$ and we have a family of \emph{DHR endomorphisms} given by $\{\rho_I\}_I$.

Let $\rho, \sigma \in \Rep^{I}(\A)$ and $I_1, I_2$ be disjoint intervals inside $I$ such that $I_1 < I_2$, in other words $I_1$ is clockwise from $I_2$ inside $I$.
Then choose $\hat{\rho} \in [\rho]$ localized in $I_1$ and $\hat{\sigma} \in [\sigma]$ localized in $I_2$. 
Also choose charge transporters $u_{\rho} \in (\rho, \hat{\rho})$ and $u_{\sigma} \in (\sigma, \hat{\sigma})$.

Define the 
the braiding by
\begin{align}
\label{DHRb}
\varepsilon(\rho, \sigma):=\sigma(u_{\rho}^{\ast})u_{\sigma}^{\ast}u_{\rho}\rho(u_{\sigma})\,,
\end{align}
which is independent of the choices of  $\hat{\rho}, \hat{\sigma}$ as well as $u_\rho, u_\sigma$.
From the property $\hat{\rho}\circ \hat{\sigma}=\hat{\sigma} \circ \hat{\rho}$, we have $\varepsilon(\rho, \sigma) \in (\rho \sigma, \sigma \rho)$.

Moreover, if we choose opposite order of intervals $I_2 < I_1$, then we have the opposite braiding associated to $\varepsilon(\rho, \sigma)$ is given by
\begin{align*}
\varepsilon_{\rho, \sigma}^{-}:=\varepsilon_{\rho, \sigma}^{\ast}
\end{align*}
and we write $\varepsilon_{\rho, \sigma}^{+}:=\varepsilon_{\rho, \sigma}$.
Therefore, the category $\Rep^{I}(\A)$ admits \emph{DHR braiding} defined in Equation~\eqref{DHRb}.
\subsection{$\ZZ/{2\ZZ}$--kernel }
\label{ssec:exmpleoforder2}
Let $G=\langle \chi \rangle \cong \ZZ/{2\ZZ}$ for $\chi \in \Out(M)$ with $\chi \circ \chi=[\ide_M]$.
Let $\alpha \in \Out(M)$ such that $[\alpha]=\chi$.
Then by definition there exists an $u \in \U(M)$ such that $\alpha^2=\Ad(u)$. 
Note that for all $x \in M$, we have
\begin{align}\label{auteq}
   \nonumber &\alpha^2(x)=uxu^{\ast}\\
    &\alpha(\alpha(x))=uxu^{\ast}\,.
    \end{align}
    Applying $\alpha^{-1}$ on the equation \eqref{auteq} gives
    \begin{align*}
    &\alpha(x)=\alpha^{-1}(u) \alpha^{-1}(x) \alpha^{-1}(u)^{\ast}\,.
    \end{align*}
    Furrthermore, evaluating $\alpha^{-1}(x)$ on the equation \eqref{auteq} also gives
    \begin{align*}
    &\alpha(x)=u \alpha^{-1}(x) u^{\ast}\,.
    \end{align*}
    Equating both sides, we have
    \begin{align*}
    &\alpha^{-1}(u) \alpha^{-1}(x) \alpha^{-1}(u)^{\ast}= u \alpha^{-1}(x) u^{\ast}\,.
    \end{align*}
    As $\alpha$ is a homomorphism and by applying $\alpha$ on both sides of the above equation, we obtain
    \begin{align*}
     u x u^{\ast}=\alpha(u)x \alpha(u)^{\ast}\,.
    \end{align*}
    Hence, we get
    \begin{align*}
     \alpha(u)^{\ast}u x u^{\ast} \alpha(u)=x\,.
\end{align*}
Therefore, $u^{\ast}\alpha(u) \in Z(M)=\CC 1_{M}$ and consequently $\alpha(u)=\nu u$ for some $\nu \in \CC$.
Since $\alpha^{2}(u)=u$, we have $\nu^{2}=1$.
As $\nu=\pm 1$, it corresponds to \emph{Frobenius-Schur} indicator which appears in representation theory.
The case $\nu=-1$ is called twisted, while $\nu=1$ is called untwisted.
In particular, $\nu \in H^{3}_{\Borel}(\ZZ/2\ZZ, \U(1)) \cong \ZZ/2\ZZ$ is the obstruction for $\chi$.

\subsection{Bockstein Homomorphisms}
\label{ssec:bock}
Let $X$ be a topological space.
Given a short exact sequence of abelian groups
\begin{align*}
0 \longrightarrow G \longrightarrow H \longrightarrow K \longrightarrow 0\,,
\end{align*}
we have a exact sequence of chain complexes for all $n \in \NN$ given by 
\begin{align*}
0 \longrightarrow C^{n}(X, G) \longrightarrow C^{n}(X,H) \longrightarrow C^{n}(X, K) \longrightarrow 0\,.
\end{align*}
Then for each $n \in \NN$ there is a long exact sequence of  cohomology groups given as
\begin{align*}
\cdots \longrightarrow H^{n}(X, G) \longrightarrow H^{n}(X, H) \longrightarrow H^{n}(X, K) \longrightarrow H^{n+1}(X, G) \longrightarrow \cdots 
\end{align*}
The \emph{boundary map} (also known as \emph{connector map}) $\delta_n \colon H^{n}(X, K) \rightarrow  H^{n+1}(X, G)$ is called a \emph{Bockstein Homomorphism} for every $n \in \NN$.
\subsection{Local nets of standard subspaces}
\label{ssec:Local nets of standard subspaces}
Let $F$ be a finite-dimensional Euclidean space of dimension $n$ with an inner product
$\langle \slot, \slot \rangle_F$.
We denote the set of all smooth functions from the circle $\SOne$ to $F$ by
\begin{equation}
  L(F):=C^{\infty}(S^1, F)\,.
\end{equation}
A function $f$ $\in$ $L(F)$ can be represented by its \emph{Fourier series} with
\emph{Fourier coefficients} $\{\hat{f}_k\}_{k\in\ZZ}$ in the the complexification
$F_\CC := F \otimes_{\RR}\CC$ of $F$ as follows
\begin{align}
  f(\theta) & = \sum_{k \in \ZZ}\hat{f_k}e^{\ima k \theta}\,,
    & \hat{f_k} =  \overline{\hat f_{-k}}
    & :=\int_{0}^{2 \pi}f(\theta) e^{-\ima k \theta}\frac{d\theta}{ 2 \pi}\in F_\CC\,.
\end{align}
We introduce a seminorm $\|\slot\|$ on $L(F)$ via
\begin{equation}
  \|f\|^{2} := \sum_{k=1}^{\infty}k \|\hat{f_{k}}\|^{2}\,.
\end{equation}
The linear endomorphism $\mathfrak{J}$ on $L(F)$ given by
$\mathfrak{J}(\hat{f_k})=-\ima\sign(k) \hat{f_k}$ 
defines a complex structure on $L(F)$, namely with $\mathfrak{J}^{2}=-1$. 
Note also that
\begin{align*}
  \langle f, \mathcal{J}g\rangle 
    &= \ima\langle f, g\rangle 
    = -\langle \mathcal{J}f, g\rangle\,,
\end{align*}
and that $\|f\| = 0$ if and only if $f \in F$ is a constant function.
This gives the quotient space $L(F)/F$ the structure
of a pre-Hilbert space and we  denote its completion by $\Hil_{0, F}$.

It is straight-forward to check that the canonical symplectic form 
$\omega(\xi, \eta) = \Ima \langle\xi, \eta\rangle$
associated with the Hilbert space $\Hil_{0, F}$ is explicitly given by
\begin{equation}
  \omega([f], [g])
    = \frac{1}{2}\int_{0}^{2 \pi}
    \langle f(\theta), g'(\theta) \rangle_{F} \frac{d\theta}{2 \pi}
    =: \frac{1}{2}\int \langle f, g' \rangle_F\,.
\end{equation}
Define
\begin{align*}
  L_{I}(F)=\{f \in L(F): \supp(f) \subset I\}\,.
\end{align*}
For any open connected proper interval $I$ of $S^1$, the nets of Hilbert spaces
are given by
\begin{align*}
  H_{F}(I)=\overline{\{f: f \in L(F), \supp(f) \subset I\}/F}\subset \Hil_{0, F}\,,
\end{align*}
and the corresponding conformal nets $\A_{F}(I)$ are given by
\begin{align*}
  \A_{F}(I)=\{ W([f]): f \in L_{I}(F)\}'' \subset \B(\e^{\Hil_{0,F}})\,.
\end{align*}

\subsection{Second quantization}
\label{ssec:SecondQuantization}
Given a Hilbert space $\Hil$ and consider the  associated sesquilinear form $\omega$ 
defined by $\omega([x],[y])=\Ima \langle [x],[y]\rangle_\Hil$ for all $x, y$ in the pre-Hilbert space of $\Hil$. 
We get a representation of the Weyl algebra on the Bosonic Fock space $e^{\Hil}$ 
given by
\begin{align}
    W(\xi)e^{0} &= e^{-\frac{1}{2}\|\xi\|^{2}}e^{\xi}\, & \xi\in \Hil\,.
\end{align}
 Here the Bosonic Fock space $e^{\Hil}$ is given by
\begin{align*}
  e^{\Hil}:=\bigoplus_{n=0}^{\infty}P_{n} \Hil^{\otimes n}\,,
\end{align*}
where  $P = \bigoplus_{n} P_{n}$ is the projection onto the symmetric subspace
defined by
\begin{align}
  \label{eq:Projection}
  P_{0}\Omega & = \Omega\,, &
  P_{n}(x_{1} \otimes \cdots \otimes x_{n}) 
    &= \frac{1}{n!}\sum_{\sigma \in S_n}x_{\sigma(1)}\otimes \cdots \otimes x_{\sigma(n)}\,.
\end{align}
Note that the action is well defined since the set of all \emph{coherent vectors}
\begin{align}
  \label{eq:CoherentVectors}
  e^\xi &:=\bigoplus_{n=0}^{\infty}\frac{1}{\sqrt{n!}}\xi^{\otimes n}\,,
    &&(\xi \in \Hil)
\end{align}
is total in $e^{\Hil}$.
Furthermore, it can be checked that $\{\W(\xi)\}_{\xi\in\Hil}$ indeed satisfy the 
\emph{Weyl relations}
\begin{align}
  \label{eq:Comuwl}
  W(\xi)W(\eta) & =e^{-i \omega(\xi, \eta)}W(\xi + \eta)  
  =e^{-2i\omega(\xi, \eta)}W(\xi) W(\eta)\,,& \text{for all }\xi, \eta \in \Hil\,.
\end{align}
Finally, we call $\Omega:=e^{0}$ the \emph{vacuum vector} of this representation
which defines a state with $\varphi(W(\xi)) =  e^{-\frac{1}{2}\|\xi\|^{2}}$.

\section{Construction of obstruction}
\label{sec:Construction of obstruction}
\subsection{Local nets of standard subspaces}
\label{ssec:LocalNetsOfStandardSubspaces}
We briefly review some background on local nets of standard subspaces. Let
$I\subset S^{1}$ be a \emph{proper interval}, \ie open, connected, non-empty
and non-dense subset of the unit circle $S^{1}$. We denote by $I'=S^{1}\setminus
\overline{I}$ the \emph{complement interval} which is again proper.
Consider the set of all real-valued smooth functions on
$S^{1}$ denoted by $LR:=C^{\infty}(S^{1}, \RR)$. Note that any element $f$ in $C
^{\infty}(S^{1}, \RR)$ can be written as a \emph{Fourier series} exactly the same as in the general $n$-dimensional case.

We equip $C^{\infty}(S^{1}, \RR)$ exactly same seminorm
and complex structure via the operator $\mathcal{J}$ coming from the construction of $H_F$ with the special case $F=\RR$. 

 Take the quotient space $C^{\infty}(S^{1}, \RR)\slash
\RR$ by identifying constant functions on $S^{1}$ with $\RR$. On this ambient
quotient space, the seminorm induces a norm. The inner product on
$C^{\infty}(S^{1}, \RR)\slash \RR$ is given by
\begin{align*}
  \langle f, g \rangle = \sum_{k=1}^{\infty}k \hat{f_k}\hat{g_{-k}}\,,
\end{align*}
where we write $f, g$ in $C^{\infty}(S^{1}, \RR)\slash \RR$ by abuse of
notation. By completing the quotient space $C^{\infty}(S^{1}, \RR)\slash \RR$ with
respect to the norm, we get as a byproduct
$\Hil=\overline{C^{\infty}(S^1, \RR)\slash \RR}^{\|\cdot\|}$ with 
the complex structure is given by $\mathcal{J}$ and similar as the construction of $H_F$.

It comes with a natural symplectic form given by
\begin{align}
  \label{eq:sesquilinear_form}\omega([f], [g]) = \Im \langle f, g\rangle = -\frac{i}{2}\sum_{k \in \ZZ}k\hat{f_k}\hat{g_{-k}}= \frac{1}{2}\int_{0}^{2 \pi}f(\theta) g'(\theta) \frac{d\theta}{2 \pi}\,,
\end{align}
for all $f, g \in L_{I}(\RR)$.
For a given proper interval $I$ of $S^1$, we define
\begin{align}
  \label{prelp nets}\L_{I}\RR=\{ f \in L\RR: \supp(f) \subset I\}\,.
\end{align}
We introduce a seminorm $\|\slot\|$ on $L(\RR)$ via
\begin{equation}\label{norm}
  \|f\|^{2} := \sum_{k=1}^{\infty}k \|\hat{f_{k}}\|^{2}\,, \qquad f \in L(\RR)\,.
\end{equation}
Define the net of Hilbert spaces to be
\begin{align*}
  H_{\RR}(I)=\overline{L_{I}\RR \slash \RR}^{\|\cdot\|}\subset \Hil\,,
\end{align*}
where the norm is given in Equation~\eqref{norm}. Observe that $\{H_{\RR}(I)\}_{I\in\cI}$
is a local (and M\"obius covariant) net of standard real Hilbert spaces by \cite[Section~3.1]{Bi2011}.

Define the associated $\U(1)$-current net
\begin{align}
  \label{CCR}\A_{\RR}(I)=\{W([f]): f \in L_{I}\RR\}'' \subset \mathcal{B}(e^{\Hil})\,,
\end{align}
for any proper interval $I$ of $S^{1}$. Here, notationally for a set $S \in \mathcal{B}
(\H)$, $S''$ denotes the von Neumann algebra generated by $S$. The association
$I \mapsto \A(I)$ defines a local conformal net $\A$, see \eg
\cite[Proposition~3.5]{Bi2011} for details. By \cite{DR1975} and \cite{Lo1979}
$\A(I)$ is a hyperfinite type III${}_{1}$ factor for any $I\in \cI$.

\subsection{Gauge automorphisms of the $\U(1)$-current}
\label{ssec:GaugeAuto} 
We identify smooth functions on $S^1$ as $2 \pi$ periodic functions on $\RR$, \ie any $f \in C^{\infty}(S^1, \RR)$ satisfies $f(\theta+2\pi)=f(\theta)$ for all $\theta \in \RR$ via the identification map $\RR \ni \theta \mapsto e^{\ima \theta} \in \SOne$.
By an abuse of notation, we will write $\int _{S^1} f=\int_{\theta=0}^{2 \pi} f(\theta) \frac{d \theta}{2 \pi}$.
Fix a proper interval $I$ and a smooth function
$m \in C^{\infty}(S^{1}, \RR)$ such that $\supp(m) \subset I$. 
Let $f$ be a smooth function supported in $\tilde{I}\supset I$ such that
$f|_{I}=1$.
We regard $f$ as a $2 \pi$ periodic function and let $t \in \RR$ with $f(t)=0$.
We define 
\begin{align}\label{contm}
L(x)=f(x) \int_{t}^{x}
m(y) dy\,.
\end{align}
Here we choose $x, t \in \RR$ such that the image of the interval $(t, x)$ is contractible, \ie lies inside the cut circle $S^1 \setminus \{e^{it}\}$.  Note that because of the support of $f$ this is well-defined and independent of the particular choices and that $f$ is continuous and thus smooth on the circle.
Define
\begin{align}
  \label{defngaug}\alpha_{m}=\Ad W([L])\,.
\end{align}
While this follows indirectly, it is enlightening to check directly that $\Ad W([L]) W([f])$ with $[f] \in H(I)$ does really only depend on the equivalence classes $[L]$ and $[f]$ and not on any of the choices.
\begin{rmk}\label{classindependency}
Let $I=(a, b)\subset S^1$.
Fix $d \in \RR$ and let $c \in \RR$.
Let $f, g, m, $ as defined above and $M$ an antiderivative of $m$ on $S^1\setminus \{p\}$ for some point $p\not\in \hat I$. 
Define $h(x)=g(x)+c$ and $L(x)=M(x)f(x)+d$ with $g(a)=g(b)=M(a)=0$ and $M'=m$ such that $\supp(g)$ and $\supp(m)$ in $(a, b)$ and $\supp(f) \in (a-e, b+e):=\tilde{I} \subset \SOne$ with $f|(a,b)=1$.
Then
$\omega(h, L)=\omega(g, L)$.
The proof is as follows.
As $L$ is a periodic function, we have $\int_{S^1}L'=0$.
Let us  consider
\begin{align*}
\omega(h, L)&=\frac{1}{2} \int_{S^1} (h(\theta)+c)L'(\theta) \frac{d\theta}{2 \pi}\\
&=\frac{1}{2}  \int_{I} g(\theta)m(\theta)\frac{d\theta}{2 \pi}  + c\int_{S^1} L'(\theta) \frac{d\theta}{2 \pi}\\
&=\frac{1}{2}  \int_{S^1} g(\theta)m(\theta)\frac{d\theta}{2 \pi}\\
&=\omega(g, L)\,.
\end{align*}
\end{rmk}
\begin{prop}
  \label{GauAut} The formula of the automorphism $\alpha_{m}$ on the Weyl algebra
  is given by
  \begin{align}
  \label{Gaut}
    \alpha_{m}(W([g]))=e^{i\int_{0}^{2 \pi} g(\theta)m(\theta) \frac{d \theta}{2 \pi}}W([g])\,.
  \end{align}
  In particular, $\alpha_{m}$ is an automorphism on $\A_{\RR}(I)$ and independent of
  the choice of $f$.
\end{prop}
\begin{proof}
  Let $L$ be as in Lemma~\ref{classindependency}.
   Then evaluating $\alpha_{m}$ on the Weyl unitary
  $W([g])$, we get
  \begin{align*}
    \Ad_{W([L])}(W([g])) & = W([L]) W([g]) W([-L])                                                     \\
                         & =e^{-i \omega (L, g)}W([L+g])W([-L])                                        \\
                         & =e^{-i \omega (L, g)- i\omega(L+g, -L)}W([g])                               \\
                         & =e^{-i \omega(L, g)+ i w(g, L)}W([g])                                       \\
                         & =e^{2i \omega (g, L)}W([g])                                                 \\
                         & =e^{i \int_{0}^{2 \pi} g(\theta) m(\theta) \frac{d \theta}{2 \pi}}W([g])\,.
  \end{align*}
  Observe that $\alpha_{m}$ is independent of the choice of $f$. Since $\alpha_{m}$
  is unitarily implemented, we have $\alpha_{m} \in \Aut(\A_{\RR}(I))$.
\end{proof}
\begin{lem}
  \label{lem:InnerAuto} Consider a proper interval $I:=(a, b) \subset \SOne$ and a smooth function $m
  \colon S^1 \rightarrow \RR$ with $\supp m \subset I$ and $\int_{I} m=0$. Then
  $\alpha_{m}$ is an inner automorphism of $\A_{\RR}(I)$.
\end{lem}
\begin{proof}
  As $\int m=0$, we have $M(x)=f(x) \int_{a}^{x} m(y) dy= \int_{a}^{x} m(y) dy ~( f|I=1)$ and
  therefore $M$ has support in $I$. Thus $[M] \in H_{\RR}(I)$. Therefore
  $\alpha_{m}$ is an inner automorphism by definition.
\end{proof}
\begin{thm}
  \label{equivlsamecharge}
   Let $I \subset S^1$ be a proper interval and $\ell, m \in C^{\infty}(S^1, \RR)$ with support in $I$.
    Consider the automorphisms $\alpha_{\ell}$ and $\alpha_m$ on $\A_{\RR}(I)$ defined by the Equation~\eqref{Gaut}.
    If $\alpha_{\ell}$ is inner equivalent to $\alpha_{m}$,
  then $\int{\ell}=\int{m}$.
\end{thm}
\begin{proof}
  Let $u \in \U(\A_{\RR}(I))$ be a unitary in $\A_{\RR}(I)$ such that
  \begin{align*}
    u \alpha_{\ell}(x)=\alpha_{m}(x)u\,.
  \end{align*}
  By the isotony property of conformal nets, we have
  $\A_{\RR}(I) \subset \A_{\RR}(\tilde I)$ for any interval $\tilde{I}$ containing $I$. 
  We
  claim that $u$ is an intertwiner, \ie
  $u \in (\alpha_{m}, \alpha_{n})_{\A(\tilde I)}$. 
  Here $\cI$ denotes the set of proper intervals of
  $\SOne$ and $I_{1}, I_{2} \in \cI$ are disjoint intervals such that the
  disjoint union $I_{1}\sqcup{p}\sqcup I_{2}\in \cI$ for some point $p$.
  Recall that the strong additivity
  property of $\U(1)$-current net (see \cite[Proposition~3.11]{DoXu2006}) implies $\A_{\RR}(I)=\A_{\RR}(I_1) \vee \A_{\RR}(I_2)$, where $I_1$ and $I_2$ are proper intervals obtained by removing an interior point of $I$ and $\A_{\RR}(I_1) \vee \A_{\RR}(I_2)$ denotes the von Neumann algebra generated by $\A_{\RR}(I_1)$ and $\A_{\RR}(I_2)$.

  Consider the intervals
  \begin{align*}
    I = (a, b)\,, \qquad \tilde{I}=(e, c)\,, \quad I_{1} = (e, a)\,,\quad I_{2} = (b, c)\,,
  \end{align*}
  such that $e<a<b<c$ and $e, a, b, c \in \RR$. 
  Notice that $I$, $I_{1}$, and
  $I_{2}$ are pairwise disjoint. 
  As a consequence, we immediately see that $\A_{\RR}(I)$ and
  $\A_{\RR}(I_{1})$ commute elementwise.
   Likewise $\A_{\RR}(I)$ and $\A_{\RR}(I_{2})$ also commute
  elementwise. 
  We only check the commutation relation for $x \in \A_{\RR}(I_{1})$ or
  $\A_{\RR}(I_{2})$. 
  Since the automorphisms $\alpha_{\ell}$ and $\alpha_{m}$ are localized,
  it follows that
  \begin{align}
    \label{loc pr}\alpha_{\ell}(x)=\alpha_{m}(x)=x \text{ when }x \in \A_{\RR}(J)\,, & ~J \subset S^1\setminus \{-1\}\,, I\cap J =\emptyset\,.
  \end{align}
  Nonetheless, one gets the following:
  \begin{align*}
    u \alpha_{\ell}(x) & = ux            &  & (\text{by Equation~ \eqref{loc pr}})             \\
                       & =xu             &  & (\text{by locality of nets})                     \\
                       & =\alpha_{m}(x)u &  & (\text{again using Equation~ \eqref{loc pr}})\,,
  \end{align*}
  whenever $x \in \A_{\RR}(I_{1})$ or $\A_{\RR}(I_{2})$. 
  More precisely, by using the strong
  additivity, we have
  \begin{align*}
    u \alpha_{\ell}(x)= \alpha_{m}(x) u \text{ for all }x \in \A_{\RR}(\tilde{I})\,.
  \end{align*}
  We have to show that
  \begin{align}\label{uteq}
   u W([f])= W([f]) u
   \end{align}
for any smooth constant function $f$
  on $I$.
   Since $u \in \A_{\RR}(I)$, we can approximate $u$ by Weyl unitaries, and
  it is enough to check the Equation~\eqref{uteq} for $u=W([g])$ for some smooth function $g$ supported
  on $I$.
   Since $\supp(g) \subset I$ and $f'\equiv 0$ on $I$, we have
  $\omega(g, f)=0$. 
  Thus
  \begin{align*}
    W([g])W([f])=W([g+f])=W([f]) W([g])\,.
  \end{align*}
  Therefore, we get $u W([f])= W([f]) u$ for any smooth function $f$ constant
  on $I$. 
  By the relation $u \alpha_{\ell}(x)=\alpha_{m}(x) u$ and plugging into
  $x=W([f])$ in the ambient relation, we obtain
  \begin{align*}
    u \alpha_{\ell}(W([f]))                                                   & =\alpha_{m}(W([f])) u                                                          \\
    e^{{\frac{\ima}{2 \pi} \int_{0}^{2 \pi}}f(\theta) l(\theta) d \theta}uW([f]) & =e^{{\frac{\ima}{2 \pi} \int_{0}^{2 \pi}}f(\theta) m(\theta) d\theta}W([f]) u\,, & f \in H_{\RR}(\tilde{I})\,.
  \end{align*}
  By choosing $f$ with $f|_{I} \equiv c$ for any constant $c\in \RR$, we get
  \begin{align*}
    \left[e^{\frac{1}{2 \pi} \int_{I} c \ell(\theta) d \theta}- e^{\frac{1}{2 \pi} \int_{I} c m(\theta) d\theta}\right] u W([f]) & =0\,.
  \end{align*}
  Consequently, we get
  \begin{align}
    \label{expon}e^{\frac{1}{2 \pi} c\int_{I} \ell(\theta) d \theta}= e^{\frac{1}{2 \pi} c\int_{I} m(\theta) d\theta} & (c\in\RR)\,.
  \end{align}
  Since the Equation~\eqref{expon} holds for any non-zero $c\in\RR$, we
  conclude $\int_{I}\ell= \int_{I}m$. 
  Indeed, we obtain
  $\int_{0}^{2 \pi}\ell(\theta) \frac{d \theta}{2 \pi}=\int_{0}^{2 \pi}m(\theta
  )\frac{d \theta}{2 \pi}$
  because $\supp(l)$ and $\supp(m)$ are contained in $I$.
\end{proof}

Before going to the next subsection, given  $I, J  \in \cI$ contained in bigger inverval $K \in \cI$, we recall the notation $I < J$ and $I > J$.
Here $I < J$ means that $I$ is left of $J$ clockwise inside $K$ and $I > J$ means that $I$ is right of $J$ clockwise inside $K$.
\subsection{The braiding for the \texorpdfstring{$\U(1)$}{U(1)}-current net}
\label{ssec:TheBraidingOfUOneCurrentNet}
\begin{prop}
  \label{prop:BraidingFormula} Let $I_{1} < I_{2} < I_{3}$ be subintervals of a given interval $I \subset S^1$
  and let $\ell_{1},\ell_{2},\ell_{3}\colon S^{1}\to \RR$ be smooth functions with
  $\supp \ell_{i}\subset I_{i}$ and $\int_{0}^{2 \pi}\ell_{1}(\theta) \, \frac{d\theta}{2
  \pi}= \int_{0}^{2 \pi}\ell_{3}(\theta)\, \frac{d\theta}{2 \pi}$. Let
  $\rho_{i}=\rho_{\ell_i}$ and recall that
  \begin{align*}
    \rho_{i}(W([f]))=e^{\textstyle\ima\int_{0}^{2 \pi} f(\theta) \ell_i(\theta)\frac{d\theta}{2\pi} }W([f])\,.
  \end{align*}
  Then the Fredenhagen-Rehren-Schroer braiding is given by
  $\varepsilon(\rho_{1}, \rho_{2})=e^{2 \pi \ima q(\ell_1) q(\ell_2)}\cdot 1_{\A_{\RR}(I)}$,
  where $\A_{\RR}(I)$ is the algebra defined in the Equation~\eqref{CCR} and $q_{\RR}(\ell)=\int_{I}\ell$ is the associated charge and $\varepsilon(\rho_{2}
  , \rho_{1})=1$.

  Moreover, the braiding for the localized automorphisms $\rho_{p \ell}$ for
  any real $p$ is given by
  \begin{equation}
    \label{selfbraid}\varepsilon(\rho_{p\ell}, \rho_{q\ell}) = \e^{\ima \pi p
    q}\cdot 1_{\A_{\RR}(I)}\,,
  \end{equation}
  where the charge of $\rho_{\ell}$ is 1, and $p$ and $q$ are real numbers.
\end{prop}
\begin{proof}
  Let $I_{1} < I_{2} < I_{3}$ be subintervals of $I$ and let $\ell_{1},\ell_{2}
  ,\ell_{3}\colon S^{1}\to \RR$ be smooth functions with $\supp \ell_{i}\subset
  I_{i}$ and $\int_{0}^{2 \pi}\ell_{1}(\theta) \frac{d\theta}{2 \pi}=\int_{0}^{2
  \pi}\ell_{3}(\theta) \frac{d\theta}{2 \pi}$. Let $\rho_{i}=\rho_{\ell_i}$
  and recall that
  \begin{align*}
    \rho_{i}(W([f]))=e^{\textstyle\ima\int_{0}^{2 \pi} f(\theta) \ell_i(\theta)\frac{d\theta}{2\pi} }W([f])\,.
  \end{align*}

  Let $L$ be a smooth function with support in $I$ and $L'=\ell_{3}-\ell_{1}$.
  Then
  \begin{align}
    \label{subbraid}\rho_{3}\rho_{1}^{-1}= \rho_{\ell_3-\ell_1}= \Ad W([L])
  \end{align}
  on $\A_{\RR}(I)$. 
  Thus $W([L]) \in (\rho_{1},\rho_{3})$. 
  Then
  $\rho_{2}(W([L])^{\ast})W([L]) \in (\rho_{1}\rho_{2},\rho_{2}\rho_{1})$. 
  We want
  to compute the following:
  \begin{align}
    \label{cptbraid1}
    \rho_{2}(W([-L]))W([L]) 
      & = \e^{\textstyle-\ima \int_I\ell_2(\theta)L(\theta)\frac{d\theta}{2\pi}}W([L])^{\ast} W([L])                                                                                 \\
      & = \e^{\textstyle-\ima \int_I\ell_2(\theta)L(\theta)\frac{d\theta}{2\pi}}\cdot 1_{\A_{\RR}(I)}                                                                                      \\
      & = \e^{\textstyle-\ima \int_{I_2}\ell_2(\theta)L(\theta)\frac{d\theta}{2\pi}}\cdot 1_{\A_{\RR}(I)}                                                                                  \\
      & = \e^{\textstyle-\ima \int_{I_2}\ell_2(\theta)\int_{-\infty}^\theta L'(\vartheta)\,d\vartheta\frac{d\theta}{2\pi}}\cdot 1_{\A_{\RR}(I)} \\
      & = \e^{\textstyle-\ima \int_{I_2}\ell_2(\theta)\int_{I_1}L'(\vartheta)\,d\vartheta\frac{d\theta}{2\pi}}\cdot 1_{\A_{\RR}(I)}\qquad\text{(support of $L'$)}                                \\
      & = \e^{\textstyle\ima \int_{I_2}\ell_2(\theta)\int_{I_1}\ell_1(\vartheta)\,d\vartheta\frac{d\theta}{2\pi}}\cdot 1_{\A_{\RR}(I)}                                                     \\
      & = \e^{\textstyle2\pi\ima\int_{I}\ell_2(\theta)\frac{d\theta}{2\pi}\cdot \int_{I}\ell_1(\vartheta)\frac{d\vartheta}{2\pi}}\cdot 1_{\A_{\RR}(I)}                                     \\
      & = \e^{\textstyle 2\pi\ima q(\ell_1)q(\ell_2)}\cdot 1_{\A_{\RR}(I)}\,,
  \end{align}
  where
  \begin{align}
    \label{charge}q_{\RR}(\ell) = \int_{I}\ell(\theta)\frac{d\theta}{2\pi}\,.
  \end{align}
  Note that $\varepsilon(\rho_{1},\rho_{2}) = \rho_{2}(W([-L]))W([L])$ is the
  Fredenhagen--Rehren--Schroer braiding and
  $\varepsilon(\rho_{2}, \rho_{1}) = 1$, for instance look at \cite{GuLo1996}.
  In particular, we have
  \begin{align*}
    \varepsilon(\rho_{\ell}, \rho_{\ell})^{2} = \varepsilon(\rho_{1}, \rho_{2})\varepsilon(\rho_{2},\rho_{1}) = e^{2\ima\pi q(\ell)^2}\,.
  \end{align*}

  Let us fix $\ell_{1}$ in $I_{1}$ with $q_{\RR}(\ell_{1})=1$ and let us compute now
  the braiding between $\rho_{p\ell_1}$ and $\rho_{q\ell_1}$.

  Similarly, let $\ell_{2}$ be localized in $I_{2}$. Let $p$ be any real
  number. Let $L$ be as above with $L' = \ell_{2} - \ell_{1}$, then
  \begin{equation}
    \rho_{p\ell_2}\rho^{-1}_{p\ell_1}= \rho_{p(\ell_2-\ell_1)}= \Ad W([pL])
  \end{equation}
  and thus $W([pL]) \in (\rho_{p\ell_1}, \rho_{p\ell_2})$.

  Then
  $\rho_{q\ell_1}(W([pL])^{\ast}) W([pL])\in (\rho_{p\ell_1}\rho_{q\ell_1}, \rho
  _{q\ell_1}\rho_{p\ell_1})$
  is the braiding. 
  Consequently, we have
  \begin{align}\label{fnkycmpt}
    \rho_{q\ell_1}(W([-pL]))W([pL]) 
      & = \e^{\textstyle-\ima pq \int_{I_1}\ell_1(\theta)L(\theta)\frac{d\theta}{2\pi}}W([pL])^{\ast} W([pL])          \\
      & = \e^{\textstyle \ima pq\int_{I_1}\ell_1(\theta)L_1(\theta)\frac{d\theta}{2\pi}}\cdot 1_{\A_{\RR}(I)}                \\
      & = \e^{\textstyle\frac{1}{2}\ima pq\int_{I_1}\frac{d}{d\theta}L_1^2(\theta)\frac{d\theta}{2\pi}}\cdot 1_{\A_{\RR}(I)} \\
      & = \e^{\textstyle\ima \pi pq }\cdot 1_{\A_{\RR}(I)}
  \end{align}
  because $q_{\RR}(\ell_{1})=1$ , where $L_{1}'=\ell_{1}$. 
  Namely, one has
  \begin{equation}\label{chargecmpt}
    1= q(\ell_{1}) = \int_{I}\ell_{1}(\theta)\frac{d\theta}{2\pi}= \frac{1}{2\pi}
    (L_{1}(\partial^{+} I_{1}) - L_{1}(\partial^{-} I_{1})) = \frac{1}{2\pi}L_{1}
    (\partial^{+} I_{1})\,.
  \end{equation}
  Thus $L_{1}$ takes the values $0$ and $2\pi$ on the boundaries
  $\partial^{-}I_{1}$ and $\partial^{+} I_{1}$ of $I_{1}$, respectively.
Therefore the result follows.
\end{proof}
\subsection{The \texorpdfstring{$\RR$}{RR}-action on
\texorpdfstring{$\U(1)$}{U(1)}-current}
\label{ssec: R action} 
Given a proper interval $I \subset S^1$,
choose a smooth function $\ell$ on $S^{1}$ such that the
charge $q(\ell)=\int_{I} \ell(\theta) \frac{d\theta}{2 \pi}=1$ with $\supp(\ell)\subset I$.
Define a $\RR$-action
on $\A_{\RR}(I)$ given in the Equation~\eqref{CCR}  by
\begin{equation}
  \RR \ni x\mapsto [\rho_{x \ell}] \in \Out(\A_{\RR}(I))\,.
\end{equation}
Note that $\rho_{x\ell}$ is a localized automorphism on $\A_{\RR}(I)$ for any $x \in \RR$.

\subsection{General case for \texorpdfstring{$n$}{n}-dimensional torus and
\texorpdfstring{$F$}{F}-action on \texorpdfstring{$\A_F(I)$}{A_F(I)}}
\label{ssec:GeneralCaseForNDimensionalTorus} 
Let $F$ be an $n$-dimensional Euclidean space and $\A_F(I)$ be the algebra defined in the Subsection~\ref{ssec:Local extensions}.
 Define a localized action of $F$ on $\A_F(I)$ as
\begin{align}
  \label{locaut}\rho_{\ell, I}(W([f])) & = \e^{\ima \int \langle f, \ell \rangle_F}W([f])\,, & \supp(\ell) \subset I_{0}, ~ I_{0}\subset I\,,
\end{align}
where $I_{0}$ is any proper interval inside the given interval $I \in \cI$.
Note that $\rho_{\ell, I}$ is
a localized automorphism on $\A_{F}(I)$.
By abuse of a notation, we write $\rho_{\ell}$ for $\rho_{\ell, I}$.
\begin{rmk}
A remark here is that one can define for every proper interval $\tilde I \supset I$ by the same formula.
The cases are 
\begin{itemize}
\item $\tilde{I} \supset I$.
\item $\tilde{I} \cap I$ is connected which directly follows from previous case by doing restriction.
\item $\tilde{I} \cap I$ is disconnected.
\end{itemize}
\textbf{Case 1}:[No overlap] 
Let $I \subseteq \tilde{I}$.
Note that there exists a proper interval $J \Subset \tilde{I}^c$. 
Let $f$ be a smooth non-negative function with $f|_{I}=1$ and $f|_{J}=0$.
Take $M(x)=f(x) \int_{\theta_0}^{x} m(x) dx$, where $p =\e^{\ima \theta_0} \in J$.
Then $\rho_{\hat{I}}=\rho_{\tilde{I}}|_{\A_{\RR}(\hat{I})}, \hat{I} \subset \tilde{I}$ defines a compatible family of endomorphisms.

\textbf{Case 2}:
If $\tilde{I} \cap I:=I_0$ is connected, then $\rho_{\tilde{I}}|_{I_0}$ defines a compatible family of endomorphisms as $I_0 \subset \tilde{I}$ and we can apply case 1.

\textbf{Case 3} [Disconnected Union]:
Let $I \cap \tilde{I} \neq \emptyset$ and disconnected.
Consider intervals $I_1$ and $I_2$ obtained by removing the point $\{p\}$ from $\tilde{I}$, where $p \in I'$ 
Then we have well-defined endomorphisms $\rho_{\tilde{I_1}}$ and
$\rho_{\tilde{I_2}}$ by case 1. 
Then define the endomorphisms as an extension of
\begin{align*}
\rho_{I}(x)&=\rho_{I_1}(x) \text{ if } \rho_{I_1}|_{\A_{\RR}(\tilde{I_2})}=\id_{\A_{\RR}(\tilde{I_2})}\\
&=\rho_{I_1}(x) \text{ if } \rho_{I_1}|_{\A_{\RR}(\tilde{I_1})}=\id_{\A_{\RR}(\tilde{I_1})}\,.
\end{align*}
Note that the ambient endomorphism extends to $S^1$ by the strong additivity property of the $\U(1)$-current net.
\end{rmk}
\begin{rmk}
  Given an even lattice $L \subset F$, we want to show there exists a cocycle
  $b$ such that $\A_{F}(I) \subset \A_{F}(I) \rtimes_{b}L$ is a local extension. 
  Recall local
  extension states that for any two disjoint connected open intervals $I_{1}$ and $I_{2}$
  of $S^{1}$, the associated local algebras $\A(I_{1})$ and $\A(I_{2})$
  commute, \ie $\A(I_{1}) \subset \A(I_{2})'$.
  For this we choose an ordered basis $(e_{1}
  ,\ldots, e_{n})$ and choose $b\colon L \times L \to \{\pm 1\}\leq \U(1)$ by $\ZZ$-bilinearly
  extending
  \begin{align*}
    b(e_{i},e_{j}) & = \begin{cases}(-1)^{\langle e_i,e_j\rangle_F}&\text{if }i\leq j\\ 1&\text{otherwise}\end{cases}
  \end{align*}
  which is a bicharacter and a 2-cocycle, \ie $b\in Z^{2}(L,\U(1))$. Note that
  \begin{align*}
    b(y,z)b(x+y,z)^{\ast} b(x,y+z) b(x,y)^{\ast} = (-1)^{0}=1\,.
  \end{align*}
  In particular, one has
  \begin{align*}b(e_{i},e_{j})b(e_{j},e_{i})^{\ast} = (-1)^{\langle e_i,e_j\rangle_F}\,.
  \end{align*}
  Here $\A_{F}(I) \rtimes_{b} L:=\langle \iota(\A_F(I)), \psi_{i}: i=1, \ldots, n \rangle$.
   Then
  we have
  \begin{align*}
    \psi_{i}\psi_{j} = b(e_{i},e_{j})\psi_{i+j}= b(e_{i},e_{j})b(e_{j},e_{i})^{\ast} \psi_{j}\psi_{i} = (-1)^{\langle e_i,e_j\rangle_F}\psi_{j}\psi_{i}\,.
  \end{align*}
  Thus
  \begin{align*}
    \A_{F}(I) \subseteq \A_{F}(I)\rtimes_{b} L
  \end{align*}
  is a local extension. Note that $\rho_{\ell, I}\in \Aut(A_{F}(I))$. Let us choose
 a fixed smooth function $\ell_{1}\colon S^{1}\to \RR$ with support in $I_{0}$ with charge $q(\ell_1)=\int \ell_1=1$.
  Then, for any $q\in F$, we choose
  $\ell_{q}=\ell_{1}\cdot q\colon S^{1}\to F$ to be the representative
  function for the sector $\rho_{\ell_q}$ of charge $q$, namely $q_{\ell_q}=q$.
  In particular, this defines an $F$-action $F\to \Aut(\A(I_{0}))$ defined by $F
  \ni q\mapsto \rho_{\ell_q}$ rather than only an $F$-kernel $q\mapsto [\rho_{\ell_q}
  ]$. We have
  \begin{align*}
    \rho_{\ell_q}\rho_{\ell_r}= \rho_{\ell_{q+r}}
  \end{align*}
  and the braiding is given by
  \begin{align}
    \label{Braidndim}\varepsilon(\rho_{\ell_q}, \rho_{\ell_r})=\e^{\ima \pi \langle q, r \rangle_F}\,.
  \end{align}
  If $\langle L, L\rangle \in 2\ZZ$, we can do the crossed product by $L$. In
  general, we need to do a twisted crossed product to get the
  locality. This is called a Klein twist in physics literature.
\end{rmk}

\subsection{Braiding for the $n$-dimensional case}
\label{ssec:BraidingForNDimensionalCase}
\begin{thm}
  \label{braidndim} Let $F$ be an $n$-dimensional Euclidean space endowed with
  inner product $\langle \cdot, \cdot \rangle_{F}$. 
  Let
  $I_{1} < I_{2} < I_{3}$ be subintervals of $I$ and let
  $\ell_{1},\ell_{2},\ell_{3}\colon S^1 \to F$ be smooth functions with
  $\supp \ell_{i}\subset I_{i}$ and
  $\int_{0}^{ 2 \pi}\ell_{1}(\theta) \, \frac{d \theta}{2 \pi}= \int_{0}^{2
  \pi}\ell_{3}(\theta)\, \frac{ d \theta}{2 \pi}$.
   Let $\rho_{i}=\rho_{\ell_i}$
  and recall that
  \begin{align*}
    \rho_{i}(W([f]))=e^{\textstyle\ima\int_{0}^{2 \pi} \langle f(\theta), \ell_i(\theta)\rangle_{F}\frac{d\theta}{2\pi} }W([f])\,.
  \end{align*}
  Then the Fredenhagen-Rehren-Schroer braiding is given by $\varepsilon(\rho_{1}
  , \rho_{2})=e^{2 \pi \ima \langle q(\ell_1), q(\ell_2) \rangle_{F}}\cdot 1_{\A_{F}(I)}$,
  where $q(\ell)=\int_{S^1}\ell$ is the associated charge and
  $\varepsilon(\rho_{2}, \rho_{1})=1$.

  Moreover, the braiding for the localized automorphisms $\rho_{p \ell}$ for any
  $p \in F$ is given by
  \begin{equation}
  \label{braidtu}
    \varepsilon(\rho_{p\ell}, \rho_{q\ell}) = \e^{\ima \pi \langle p, q \rangle_{F}}
    \cdot 1_{\A_F(I)}\,,
  \end{equation}
  where the charge of $\ell \colon S^1 \to \RR$ is given by $q_{\RR}(\ell)=1$ and furthermore it  implies that $q_F(f \ell)=f$ for any $f \in F$ and $p, q$ in $F$ .
\end{thm}

\begin{proof}
 Let $I_{1} < I_{2} < I_{3}$ be subintervals of $I$ and let $\ell_{1},\ell_{2}
  ,\ell_{3}\colon S^{1}\to F$ be smooth functions with $\supp \ell_{i}\subset
  I_{i}$ and $\int_{0}^{2 \pi}\ell_{1}(\theta) \frac{d\theta}{2 \pi}=\int_{0}^{2
  \pi}\ell_{3}(\theta) \frac{d\theta}{2 \pi}$. Let $\rho_{i}=\rho_{\ell_i}$
  and recall that
  \begin{align*}
    \rho_{i}(W([f]))=e^{\textstyle\ima\int_{0}^{2 \pi} \langle f(\theta),  \ell_i(\theta) \rangle_F \frac{d\theta}{2\pi} }W([f])\,.
  \end{align*}

  Let $L$ be a smooth function with support in $I$ and $L'=\ell_{3}-\ell_{1}$.
  Then
  \begin{align}
    \label{sbraidgn}\rho_{3}\rho_{1}^{-1}= \rho_{\ell_3-\ell_1}= \Ad W([L])
  \end{align}
  on $\A(I)$. 
  Thus $W([L]) \in (\rho_{1},\rho_{3})$ and
  $\rho_{2}(W([L])^{\ast})W([L]) \in (\rho_{1}\rho_{2},\rho_{2}\rho_{1})$. 
  Compute the following:
  \begin{align}
    \label{cptbragn}
    \rho_{2}(W([-L]))W([L]) 
      & = \e^{\textstyle-\ima \int_I\langle\ell_2(\theta), L(\theta) \rangle_F\frac{d\theta}{2\pi}}W([L])^{\ast} W([L])                                                                                 \\
      & = \e^{\textstyle-\ima \int_I \langle \ell_2(\theta), L(\theta)\rangle_F\frac{d\theta}{2\pi}}\cdot 1_{\A_{F}(I)}                                                                                      \\
      & = \e^{\textstyle-\ima \int_{I_2}\langle \ell_2(\theta), L(\theta) \rangle_F\frac{d\theta}{2\pi}}\cdot 1_{\A_{F}(I)}                                                                                  \\
      & = \e^{\textstyle-\ima \int_{I_2} \langle \ell_2(\theta), \int_{-\infty}^\theta L'(\vartheta)\,d\vartheta \rangle_F \frac{d\theta}{2\pi}}\cdot 1_{\A_{F}(I)} \\
      & = \e^{\textstyle-\ima  \int_{I_2}\langle \ell_2(\theta),\int_{I_1}L'(\vartheta)\,d\vartheta \rangle_F \frac{d\theta}{2\pi}}\cdot 1_{\A_{F}(I)}\qquad\text{(support of $L'$)}                                \\
      & = \e^{\textstyle\ima  \int_{I_2}{\langle \ell_2(\theta),\int_{I_1}\ell_1(\vartheta)\,d\vartheta \rangle}_F\frac{d\theta}{2\pi}}\cdot 1_{\A_{F}(I)}                                                     \\
      & = \e^{\textstyle2\pi\ima\int_{I}\langle \ell_2(\theta),  \int_{I}\ell_1(\vartheta)\frac{d\vartheta}{2\pi} \rangle_{F}\frac{d\theta}{2\pi}}\cdot 1_{\A_{F}(I)}                                     \\
      &= \e^{\textstyle2\pi\ima\langle \int_I \ell_2(\theta)\frac{d\theta}{2\pi},  \int_{I}\ell_1(\vartheta)\frac{d\vartheta}{2\pi}\rangle_F}\cdot 1_{\A_{F}(I)}                                     \\
      & = \e^{\textstyle 2\pi\ima\langle q(\ell_1),q(\ell_2)\rangle_F}\cdot 1_{\A_{F}(I)}\,,
  \end{align}
  where
  \begin{align}
    \label{chargegn}q_{F}(\ell) = \int_{I}\ell(\theta)\frac{d\theta}{2\pi}\,.
  \end{align}
  Note that $\varepsilon(\rho_{1},\rho_{2}) = \rho_{2}(W([-L]))W([L])$ is the
  Fredenhagen--Rehren--Schroer braiding and
  $\varepsilon(\rho_{2}, \rho_{1}) = 1$, for instance look at \cite{GuLo1996}.
  More precisely, it is exactly the braiding formula in \cite{FrReSc1989} given by 
  \begin{align*}
   \varepsilon_{\rho_1, \rho_2}= \rho_2(U_1^{-1})U_2^{-1}U_1\rho_1(U_2) 
  \end{align*}
  with $U_1=W([L])$ and $U_2=1$.
  In particular, we have
  \begin{align*}
    \varepsilon(\rho_{\ell}, \rho_{\ell})^{2} = \varepsilon(\rho_{1}, \rho_{2})\varepsilon(\rho_{2},\rho_{1}) = e^{2\ima\pi q(\ell)^2}\,.
  \end{align*}
  
  Let us fix $\ell_{1} \colon S^1 \to \RR$ such that $\supp{\ell_1} \subset I_{1}$ with charge $q_{\RR}(\ell_1)$ and let us compute now
  the braiding between $\rho_{p\ell_1}$ and $\rho_{q\ell_1}$.

  Similarly, let $\ell_{2} \colon S^1 \to \RR$ be localized in $I_{2}$. Let $p\in F$. Let $L \colon S^1 \to \RR$ such that $L' = \ell_{2} - \ell_{1}$, then
  \begin{equation}
    \rho_{p\ell_2}\rho^{-1}_{p\ell_1}= \rho_{p(\ell_2-\ell_1)}= \Ad W([pL])
  \end{equation}
  and thus $W([pL]) \in (\rho_{p\ell_1}, \rho_{p\ell_2})$.

  Then
  $\rho_{q\ell_1}(W([pL])^{\ast}) W([pL])\in (\rho_{p\ell_1}\rho_{q\ell_1}, \rho
  _{q\ell_1}\rho_{p\ell_1})$ implies that it is an intertwiner in $(\rho_{p\ell_1}\rho_{q\ell_1}, \rho_{q\ell_1}\rho_{p \ell_1})$ and
 is indeed a braiding because it is a special case of \cite{FrReSc1989}.
  Consequently, we have
  \begin{align}\label{fnkygn}
    \rho_{q\ell_1}(W([-pL]))W([pL]) 
      & = \e^{\textstyle-\ima  \int_{I_1}\langle q \ell_1(\theta), pL(\theta) \rangle_F\frac{d\theta}{2\pi}}W([pL])^{\ast} W([pL])\\
      & = \e^{\textstyle-\ima  \int_{I_1} \ell_1(\theta) L(\theta) \langle p, q \rangle_F\frac{d\theta}{2\pi}}W([pL])^{\ast} W([pL])\\
      & = \e^{\textstyle\frac{1}{2}\ima \langle p, q\rangle_F\int_{I_1}\frac{d}{d\theta}L_1^2(\theta)\frac{d\theta}{2\pi}}\cdot 1_{\A_{F}(I)}\\
      & = \e^{\textstyle\ima \pi \langle p, q \rangle_F }\cdot 1_{\A_{F}(I)}
  \end{align}
  because $q_{\RR}(\ell_{1})=1$, where $L_{1}'=\ell_{1}$. 
\end{proof}
\subsection{Local extensions of tensor products of the
\texorpdfstring{$\U(1)$}{U(1)}-current net}
\label{ssec:Local extensions}
\begin{prop}\label{localextn}

Let $F$ be a fixed  $n$-dimensional Euclidean space and $L\subset F$ be an
  $n$-dimensional even lattice with a basis $(e_{1}, \ldots, e_{n})$.
  Let $\ell_{1}$ be a smooth real valued function on $\SOne$ with charge $1$ and $\ell_{q}:
  =\ell_{1} q$, where $q \in F$. 
  Assume $\ell_{1}$ is localized in $I$, \ie $\ell_{1}|_{I'}=0$ where $I'=\SOne \setminus{\bar{I}}$. 
 The action of $F$ on $\A_{F}(I)$ is given by $F \ni q \mapsto [\rho_{\ell_q}] \in \Out(\A_{F}(I))$.
   Choose the
  2-cocycle
  \begin{equation}
    b(e_{i}, e_{j}) =
    \begin{cases}
      (-1)^{\langle e_i , e_j\rangle_{F}} & \text{if}\quad i < j \\
      1                                   & \text{else}\,.
    \end{cases}
    \,.
  \end{equation}
  Then
  \begin{equation}
    \cB(I) = \A_{F}(I)\rtimes_{b} L
  \end{equation}
  gives a local extension.
\end{prop}
\begin{proof}
  Let $\cB(I)$ be the twisted crossed product with the cocycle $b$ given by unitaries
  $\{\psi_{\ell}\}_{\ell\in L}$ implementing $\rho_{\ell}$, \ie
  \begin{align}\label{exteqn}
    \psi_{\ell} \iota(a) & = \iota(\rho_{\ell}(a))\psi_{\ell} & \text{ for all }a\in \A_{F}(I) \,.
  \end{align}
  By the definition of the twisted crossed product and $b$, we have the following
  (anti-) commutation relations:
  \begin{align}\label{antirlns}
    \psi_{\ell}\psi_{m} & = b(\ell,m) \psi_{\ell+m}                                                         \\
                        & = b(m,\ell)^{\ast} b(\ell, m)\psi_{m}\psi_{\ell}                                  \\
                        & = (-1)^{\sum_{i,j} q_{m_i}q_{\ell_j}\langle e_i, e_j\rangle_F}\psi_{m}\psi_{\ell} \\
                        & = (-1)^{\langle q_m, q_\ell\rangle_F}\psi_{m}\psi_{\ell}
  \end{align}
  where in the last step, we use that
  $\langle e_{i}, e_{i}\rangle_{F} \in 2\ZZ$. 
  Since
  $\varepsilon(\rho_{\ell},\rho_{m}) = (-1)^{\langle q_m,q_\ell\rangle_F}$ by
  Theorem~\ref{braidndim}, we get
  \begin{align}
    \psi_{\ell}\psi_{m} & = \varepsilon(\rho_{\ell},\rho_{m})\psi_{m}\psi_{\ell}\,.
  \end{align}
  Thus the extension is local by \cite[Theorem~6.8]{DeGi2018}.
\end{proof}
\begin{rmk}
 In particular, for $n=1$ the extension of $\RR$--action is local which follows from Proposition~\ref{localextn}.
\end{rmk}
\subsection{$\alpha$-induction}
\label{ssec:AlphaInduction} Let $A=\A_{F}(I)$. 
Let $L$ be a free abelian group of rank $n$ equipped with $\ZZ$-basis $\{x_1, \ldots, x_n\}$.
Let $\rho$ be a faithful action on $\A \subseteq \B(\Hil)$ determined by the automorphism group $\{\rho_i:=\rho_{x_i}: \rho \in \Aut(\A)\}$.
Consider the Hilbert space $\tilde{\Hil}=\bigoplus_{x \in L} \Hil_{x}$.
Let $\rho_i=\Ad(U_{\rho_i})$ such that the unitary $U_{\rho_i} \colon \Hil_{x} \to \Hil_{x+ x_i}$ for each $i=1,\ldots,n$.
The projective  representation of $L$ is written as
\begin{align*}
L \ni \ell=\sum_{i=1}^{n} \ell_ix_i \to U_{\ell}=b(g)\cdot U_1^{x_1}\cdots  U_n^{x_n}\,,
\end{align*}
where $b\colon L \to \U(1)$ is determined by the twist $b$.
Assume
$\iota\colon A \to A\rtimes_{\rho, b}L$ is the inclusion map of $A$ into the crossed-product by $L$ (twisted by a cocycle $b$).
Note that the unitaries implement the automorphisms on the crossed-product Hibert space, namely
\begin{align*}
  U_{\rho_i}\iota(a) = \iota(\rho_i(a))U_{\rho_i}\qquad \text{for all }a\in A
\end{align*}
and $B=\langle\iota(A),U_{\rho_i} ~:~ i=1,\ldots,n\rangle$.

We are interested on extending endomorphisms on $A$ to $B$.
Let $\rho$, $\sigma$ and $\tau$ be in $\End(A)$. 
The $\alpha$-induction
is determined by
\begin{align*}
  \alpha_{\sigma}(\iota(a)U_{\rho}) = \iota(\sigma(a))\varepsilon(\sigma,\rho)^{\ast} )U_{\rho}
\end{align*}
where $\{\varepsilon(\rho,\sigma)\}_{\rho,\sigma}$ is the braiding. 
Then we
have (write $\rho\sigma:=\rho\circ\sigma$)
\begin{align*}
  \alpha_{\sigma\tau}(U_{\rho}) = \iota(\varepsilon(\sigma\tau,\rho)^{\ast}) U_{\rho}
\end{align*}
and
\begin{align*}
  \alpha_{\sigma}\alpha_{\tau}(U_{\rho})=\alpha_{\sigma}(\iota(\varepsilon(\tau,\rho)^{\ast}) U_{\rho}) = \iota(\sigma(\varepsilon(\tau,\rho)^{\ast})\varepsilon(\sigma,\tau)^{\ast})U_{\rho}\,.
\end{align*}
From the property of the braiding, we get
\begin{align*}
  \varepsilon(\sigma\tau,\rho) = \varepsilon(\sigma,\rho)\sigma(\varepsilon(\tau,\rho))\,.
\end{align*}
It follows that
\begin{align*}
  \alpha_{\sigma\tau}=\alpha_{\sigma}\alpha_{\tau}\,.
\end{align*}
\subsection{Extension of the automorphism to the crossed product by braiding}
\label{ssec:ExtendAutomorphismViaBraiding}
 In general, for an $n$-dimensional
Euclidean space, the extension of $\sigma \in \Aut({\A_F(I)})$ ($\sigma$ is a \emph{DHR endomorphism}) to the twisted crossed
product is $\beta_{\sigma} \in \Out(\A_{F}(I) \rtimes_{b}L)$ defined as
\begin{align}
  \label{eq:GeneralExtension}\beta_{\sigma}(\iota(a)\psi_{\rho_t})=\iota(\sigma(a) \varepsilon(\sigma, \rho_t)^{\ast}) \psi_{\rho_t}\,,\qquad t \in L
\end{align}
that arises from $\alpha$-induction. 
Here $\varepsilon$ is the braiding in
Theorem~\ref{braidndim} and 
$\A_{F}(I) \rtimes_{b}L:=\langle\{\iota(a), \psi_{\rho_t}: a \in \A_F(I), t \in L\}\rangle$, where $\rho \colon F \ni t \mapsto \rho_t \in  \Out(\A_F(I))$ is a homomorphism.

\begin{prop}
  \label{Prop:Kernel} Consider a fixed proper interval $I \subset \SOne$ and let $\A_F(I)$ be the von Neumann algebra defined in Subsection ~\ref{ssec:LocalNetsOfStandardSubspaces}. 
  Let
  $\rho \colon F \rightarrow \Out(\A_F(I))$ given by $\rho(t)=\rho_{t \cdot \ell_0}$, where $\rho_{\ell_0}$ is defined in the Equation~\eqref{locaut}.
  Here we write $\rho_{t \cdot \ell_0}:=\rho_t$ for $t \in F$ by abuse of notation.
  Let $\beta
  \colon F \ni t \rightarrow \beta_{\rho_t} \in \Aut(\A_F(I) \rtimes_{b}L)$ be the map that arises from $\alpha$-induction and implemented from the braiding on $\A_F(I)$ 
on the twisted crossed product $\A_F(I) \rtimes_{b}L$ by a
  cocycle $b$ and the map $\beta$ is given by
  \begin{align*}
    \beta_{\rho_{t}}(\iota(a) \psi_{\rho_{\ell}})=\iota(\rho_{t}(a) \varepsilon(\rho_{t}, \rho_{\ell})^{\ast})\psi_{\rho_{\ell}}\,, \qquad t \in F, \ell \in L
  \end{align*}
  where $\iota \colon \A_F(I) \to \A_F(I) \rtimes_b L$ is the canonical embedding and $\{\psi_{\rho_{\ell}}: \ell \in L\}$ are the unitaries implementing the crossed product .
  Then $\ker(\beta_{\rho_{\ell}}) = L$.
\end{prop}
\begin{proof}
Let ${\B}_F(I):=\A_F(I) \rtimes_{b} L$.
  By Theorem~\ref{extcptw}, we have $\beta\in \Aut({\B}_F(I))$ as an extension of
  $\rho\in\Aut(\A_F(I))$.
   Consequently, we have
  \begin{align*}
    \beta_{\rho_{\ell}}\circ\iota =\iota\circ\rho_{\ell}, \qquad \ell \in L\,.
  \end{align*}
 Note that
  \begin{align*}
    \beta_{\rho_{\ell}}(\iota(a)\psi_{\rho_m}) & =\iota(\rho_{\ell}(a))\iota(\varepsilon(\rho_{\ell}, \rho_{m})^{*}) \psi_{\rho_m}                                                               \\
                                        & =\iota(\rho_{\ell}(a))\iota((\varepsilon(\rho_{\ell}, \rho_{m}))^{*}\psi_{\rho_m}                                                               \\
                                        & =\psi_{\rho_\ell}\iota(a)\psi_{\rho_\ell}^{*}\iota((\varepsilon(\rho_{\ell}, \rho_{m}))^{*}\psi_{\rho_m}                                        \\
                                        & =\psi_{\rho_\ell}\iota(a)\psi_{\rho_\ell}^{*}\iota((\varepsilon(\rho_{\ell}, \rho_{m}))^{*}\psi_{\rho_m}\psi_{\rho_\ell}\psi_{\rho_\ell}^{\ast} \\
                                        & =\psi_{\rho_\ell}\iota(a)\psi_{\rho_\ell}^{*}\psi_{\rho_\ell}\psi_{\rho_m}\psi_{\rho_\ell}^{*}\qquad[\text{by braiding property}]               \\
                                        & =\psi_{\rho_\ell}\iota(a)\psi_{\rho_m}\psi_{\rho_\ell}^{*}                                                                                      \\
                                        & =\Ad(\psi_{\rho_\ell})(\iota(a)\psi_{\rho_m})\,.
  \end{align*}
  Therefore $\beta_{\rho_{\ell}} = \Ad \psi_{\rho_{\ell}}$, which entails that $L\subseteq \ker(\beta_{\rho_{\ell}})$. Conversely, if $\ell \in \ker(\beta_{\rho_{\ell}})$ then $\beta_{\rho_{\ell}} =\Ad U$ for some
  $U\in \A_{F}(I)\rtimes_{b} L$. Thus
  $\iota\rho_{\ell} = \beta_{\rho_{\ell}}\iota = \Ad U \circ \iota$ and consequently
  $U\in (\iota,\iota\rho_{\ell})$ implying $\ell \in L$ by \cite{IzLoPo1998}.
\end{proof}

\begin{cor}
 Under the same hypothesis as Proposition~\ref{Prop:Kernel}, if $F=\RR$ with a trivial twist (ordinary crossed product),  
  Then $\ker(\alpha)=L$, where $\alpha$ is the $\alpha$-induction associated to $\iota \colon \A_{\RR}(I) \to \A_{\RR}(I) \rtimes L$.
\end{cor}
\begin{proof}
  The result follows from the same argument as in Proposition~\ref{Prop:Kernel}
  with a trivial twist.
\end{proof}

\section{Computation of anomalies}
\label{sec:Anomalies}
\subsection{Computing the anomaly of the \texorpdfstring{$R/L$}{R/L}-kernel}
\label{ssec:ComputingAnomaly}
We follow Jones' idea to compute an explicit formula for the three cocycle.

Let $B$ be a von Neumann algebra which is a factor with an outer action $\alpha$
of $\RR$ and assume $L=\gamma\ZZ = \ker(\alpha)$ for some $\gamma = \sqrt{2 m}$ 
for some $m\in\NN$.

Then we get a $\RR/L$--kernel given by
$\RR/L\ni x+L\mapsto [\alpha_{x}] \in \Out(B)$. 
Let us choose a $\Borel$ lift $s\colon
\RR/L\to \RR$, for example take the map defined by
$\sigma(x+L) = x$ for $x \in [0,\gamma)$.
For $q \in \RR/L$,
define $\tilde\alpha_{q}=\alpha_{s(q)}$. 
Then we have the inner equivalence 
\begin{align}
  \label{eq:JonesMethod}\tilde\alpha_{p}\tilde\alpha_{q} & =\alpha_{s(p)+s(q)}\sim \alpha_{s(p+q)}= \tilde\alpha_{p+q}
\end{align}
and we can choose a family of unitaries $\{u(p,q)\}_{p,q \in\RR/L}$ such that
\begin{align}
  \label{eq:TwoCocycleEquation}\tilde\alpha_{p}\tilde\alpha_{q} & = u(p,q)\tilde\alpha_{p+q}\,.
\end{align}
Namely, we have
\begin{align}
  \label{eq:OuterEquation}\tilde\alpha_{[p]}\tilde\alpha_{[q]} & =\alpha_{s([p])}\alpha_{s([q])}=\alpha_{s([p])+s([q])-s([p+q])}\tilde\alpha_{[p+q]}\,.
\end{align}
Since $h \in L$, there is a unitary $v_{h}\in B$ such that $\alpha_{h}=\Ad v_{h}$.
As $s([p])+s([q])-s([p+q]) \in L$ and $\alpha_{t}$ is inner for $t\in L$,
there is a unitary $v_{s([p])+s([q])-s([p+q])}$ such that $\alpha_{s([p])+s([q])-s([p+q])}
=\Ad(v_{s([p])+s([q])-s([p+q])})$. 
Thus, we can choose the $2$-cocycle
\begin{align}
  \label{eq:TwoCycleFormula}u([p], [q]) & =v_{s([p])+s([q])-s([p+q])}\,.
\end{align}

For $h, k\in L$, we have $v_{h}v_{k}=\mu(h,k)v_{hk}$ for some map $\mu\colon L
\times L \rightarrow \U(1)$.
Since $L\leq \RR$ is a normal subgroup,
we get 
\begin{align*}
\alpha_{h}=\Ad(\alpha_{g}(v_{g^{-1}hg}
 ))\,.
 \end{align*} 
Thus, there exists a map $\lambda\colon\RR\times L \rightarrow \U(1)$
 given by
\begin{align*}
  \alpha_{g}(v_{g^{-1}hg})=\lambda(g, h)v_{h}
\end{align*}
for $g\in \mathbb{R}$ and $h \in L$. 
In particular, we have
\begin{align*}
 \alpha_{g}(v_{h})=\lambda(g, h)v_{h}
\end{align*}
for $g\in \mathbb{R}$ and $h \in L$. 
Now the $3$-cocycle is given by
\begin{align}\label{eq:JonesFormula} 
   \nonumber\gamma(q, r, t)= & \mu(s(q)+s(r)-s(q+r), s(q+r)+s(t)-s(q+r+t)) \times       \\
              \nonumber  & \mu(s(r)+s(t)-s(r+t),s(q)+s(r+t)-s(q+r+t))^{\ast} \times \\
     & \lambda(s(q), s(r)+s(t)-s(r+t))
\end{align}
for any $q,r,t\in \RR/L$.
\begin{rmk}
Consider any even lattice $L \subset F$ of the $n$-dimensional Euclidean space $F$.
Then, the Equation~\eqref{eq:JonesFormula} still holds.
\end{rmk}
\subsection{Warm up: one-dimensional torus actions}
\label{ssec:WarmUpOneDimensionalTorus} Let $L\subseteq \RR$ be a one-dimensional
lattice, \ie $L =\sqrt{2m}\ZZ$ for some $m \in \NN$. 
We are computing the three cohomology
class for the obstruction of the $\RR/L$--kernel arising from the $\alpha$-induction
for the extension $\A_{\RR} \subset \A_{L}$:
\begin{align}\label{indextw}
  \theta & \colon[t] \longmapsto [\alpha_{t}]\,, & [t]\in \RR/L \cong \U(1)\,.
\end{align}
We choose a reference interval $I_{0}$ and $\{\rho_{q}\}_{q\in
L}$ localized automorphisms in $I_{0}$.
The local net associated to the lattice $L$ is defined by
\[
  \A_{L}(I_{0}) := \A_{\RR}(I_{0})\rtimes L = \langle \iota(\A_{\RR}(I_{0})) ,\psi
  _{q} : q\in L \rangle
\]
where $\psi_{q} \in (\iota,\iota\rho_{q})$ and
\[
  \A_{L}(I) := \langle \iota(\A_{\RR}(I)), \iota(z_{q,I_0,I})\psi_{q} : q\in L\rangle
\]
with unitaries $z_{q,I_0,I}\in (\rho_{q},\rho^{I}_{q})$ and $\rho^{I}_{q}$ is
localized in $I$.
With slight abuse of notation, we write $\rho$ for $\rho_{I}$.
Denote $\iota\colon \A_{\RR}(I)\to\A_L(I)$ the canonical inclusion and 
by $\{\psi_\ell\}_{\ell\in L}$ the unitaries implementing the crossed product.
In particular, we have 
$\psi_\ell\iota(x) = \iota(\rho_t(x))\psi_\ell$ for all $x\in \A_{\RR}(I)$.
There is a map $\lambda\colon \RR\times L \to \U(1)$ defined by
\begin{align}
  \label{eq:AlphaInduction}\lambda^{-}(t, \ell)\psi_{\rho_\ell} & := \alpha^{+}_{t}(\psi_{\rho_\ell}) = \iota\left(\varepsilon(\rho_{t},\rho_{\ell})^{\ast}\right)\psi_{\rho_\ell} \\
  \lambda^{+}(t, \ell)\psi_{\rho_\ell}                          & := \alpha^{-}_{t}(\psi_{\rho_\ell}) = \iota\left(\varepsilon(\rho_{\ell},\rho_{t})\right)\psi_{\rho_\ell}
\end{align}
for $\alpha^{\pm}$-induction, where $\varepsilon$ is the associated braiding on $\A_{\RR}$. 
To simplify the notation above, we use
$\alpha^{-}$-induction and call $\lambda = \lambda^{+}$. 
Thus, we have
\begin{align*}
  \lambda(t,\ell) = \iota(\varepsilon(\rho_{\ell},\rho_{t})) = \e^{\ima\pi q_tq_{\ell}}
\end{align*}
given by Proposition~\ref{prop:BraidingFormula}, where $q_t$ and $q_{\ell}$ are charges of $\rho_t$ and $\rho_{\ell}$, respectively.
Let us choose a Borel section $s_{m}$
of the quotient map $\RR\to \RR/\sqrt{2m}\ZZ$ by asking $s(x) \in [0,\sqrt{2m})$.
With this we choose a lift of $\theta$ given by
\begin{align}
  \tilde\alpha & \colon p \longmapsto \alpha_{s(p)}\,, & p \in \RR/L\,.
\end{align}
Then, we get
\begin{align}
  u(p,q)        & = v_{m}^{s_m(p)+s_m(q)-s_m(p+q)}                                               \\
  \omega(p,q,r) & = v_{m}^{-s_m(q)-s_m(r)+s_m(q+r)}\alpha_{s(p)}(v_{m}^{s_m(q)+s_m(r)-s_m(q+r)}) \\
  \omega(p,q,r) & = \e^{\mp i\pi s_m(p)(s_m(q)+s_m(r)-s_m(q+r))}\,.
\end{align}
Identifying $\U(1)\cong\RR/\ZZ\to \RR/\sqrt{2m}\ZZ$ given by $[x] \mapsto [\sqrt{2m}
x]$, we get
\begin{align}
  \omega_{\mp m}(x,y,z) & = \e^{\mp 2 \ima \pi m s(x)(s(y)+s(z)-s(y+z))} & x,y,z \in\RR/\ZZ, ~s(x)\in [0,1)
\end{align}
and in particular $\omega_{m}\omega_{n} = \omega_{m+n}$. Thus, we have proven:
      \begin{thm}
   \label{thm:CocycleOneDimensional} Let $m\in \NN$. 
   Consider the one-dimensional
   lattice $L_{m}=\sqrt{2m}\ZZ\subset \RR$ and the associated conformal net
   $\A_{L_{m}}:=\A_{\RR}(I)\rtimes_{\rho} L_m$, where $\A=\{\A_{\RR}(I): I \in \cI\}$ is the $\U(1)$-current net and the action of $\RR$ on $\A$ is given by $\rho$ in the Subsection~\ref{ssec: R action}.
   Then the $\RR$-action coming from $\alpha^{\pm}$-induction from $\A$
   onto $\A_{L_{m}}$ gives a $\RR/L_{m}$--kernel with obstruction realized by the cocycle
   \begin{align}
     \label{eq:ThreeCycle}
     \omega_{m}(x,y,z) & = \e^{\mp \pi \ima m s(x)(s(y)+s(z)-s(y+z))} & x,y,z \in\RR/\ZZ, ~s(x)\in [0,1)
   \end{align}
   by naturally identifying $\RR/L_{m} \cong \RR/\ZZ$.
    Here $s\colon \RR/\ZZ \to \RR$ is the Borel section given by
   \begin{align*}
   s(x+\ZZ)&=x & x \in [0,1)\,.
   \end{align*}
 \end{thm}
\begin{rmk}
  Recall that isomorphism classes of even positive lattices of rank one are $\{
  [L_{m}]\}_{m\in{\ZZ}}$ where
  \[
    L_{m} = \sqrt{2m}\ZZ\subseteq\RR
  \]
  equipped with the dot product of $\RR$. 
  The extension $\A_{L_m}\supseteq \A$ is also
  called the Buchholz--Mack--Todorov extension (BMT) $\A_{m}$ of the $\U(1)$-current
  net $\A$.
\end{rmk}
\begin{cor}
  For $m\in\NN$, consider the BMT-extension $\A_{m}=\A_{\sqrt{\ell}\ZZ}\supseteq
  \A$ of the $\U(1)$-current net $\A$ at level $\ell=2m$.

  There is $\U(1)$--kernel of solitons with obstruction as above where $\U(1)$
  is naturally identified with $\RR/L_{m}$ via the map
  $\e^{2\pi \ima t}\mapsto \sqrt{2m}t +L_{m}$. 
\end{cor}
\begin{rmk}
  The cyclic subgroup $C_{\ell}\leq \U(1)$ generated by a prime root of unity of
  order $\ell$ corresponds to $L^{\ast}_{m}/L_{m}\leq \RR/L_{m}$ and gives exactly
  the localized sectors of $\A_{m}$.
\end{rmk}

\begin{example}
  Let $L=\sqrt{2}\ZZ$ and $q = \sqrt{2}\left(\frac{1}{2}+\ZZ\right)$. Note
  that $L^{\ast} = L \oplus q= \frac{2}{\sqrt{2}}\ZZ$ is the dual lattice of
  $L$. 
  Define the section map by
  \begin{align*}
   s \colon \RR/\sqrt{2}\ZZ \ni x + \sqrt{2}\ZZ \to  s(x)=x \in \RR , \qquad x \in [0, \sqrt{2})\,.
   \end{align*}
   As $\sqrt{2}+ \sqrt{2} \ZZ= 0+ \sqrt{2}\ZZ$,  $s(\sqrt{2}+\sqrt{2}\ZZ)=0$.
  Then $s(q) = \frac{\sqrt{2}}{2}$ and $s(q+q)=s(\sqrt{2}+\sqrt{2}\ZZ) = 0$. Thus
  \begin{equation}
    \omega(q,q,q) = \e^{i\pi}= -1\,.
  \end{equation}
  Let $B$ be a type $\mathrm{III}$ factor with the given outer action $\alpha \colon \RR \ni t \to \alpha_t \in \Out(B)$.
  Note that $\tilde\alpha_{q}$ has order two in $\Out(B)$ and thus the
  obstruction is non-trivial since $[\omega]$ is the generator of
  $H^{3}(L^{\ast}/L, \U(1))\cong H^{3}(\ZZ/2\ZZ,\U(1)) \cong H^{4}(B\ZZ/2\ZZ ,\ZZ
  ) \cong \ZZ/2\ZZ$.
\end{example}

\begin{example}[Jones/Sutherland is a generalization of Connes]
Let $B$ be a factor and $p \in \NN$ and $x \in \ZZ$.
Let $\alpha \in \Aut(B)$ such that $\alpha^{x} \in \Inn(B)$ if and only if $p|x$.
 Thus $\alpha$ generates a
  $\ZZ/p\ZZ$--kernel given by $\theta \colon x\pmod p\mapsto [\alpha^{x}]\in \Out(B)$.

  Define a lift $\theta$ by $\tilde\alpha\colon \ZZ/p\ZZ \ni x \pmod p\mapsto \alpha
  ^{x}\in \Aut(B)$ for $0\leq x <p$.
  Observe that $\alpha^{p} \in \Inn(B)$. 
  Consequently, there is a $u\in B$ such
  that $\alpha^{p} = \Ad u$. 
  Let us fix $u$ and then define
  \begin{align*}
     & u(p-1,1) = u(1,p-1) = u                                     \\
     &  \tilde{\alpha}(q)\tilde{\alpha}(r) = \Ad u(q,r) \tilde{\alpha}(q+r) \text{ for } q, r \in \ZZ/p\ZZ\,.
  \end{align*}
  Define
  \begin{align*}
   \Out^{p}(B):=\{\alpha \in \Aut(B): \alpha^{x} \in \Inn(B) \text{ if and only if } p|x\}\,.
   \end{align*}  
  Notice $\alpha(u)=cu$ for some $c=\e^{2\pi i n/p}$ and
  $n\pmod p\in \ZZ/p\ZZ$ induces an isomorphism
  $H^{3}(\ZZ/p\ZZ,\U(1)) \to \ZZ/p\ZZ$. Respectively, the arrow maps in the
  diagram below are given by
  $\Out^{p}(B) \ni \alpha \mapsto \text{ob}(\alpha) \in H^{3}(\ZZ/p\ZZ,\U(1)
  )$, $\Out^{p}(B) \ni \alpha \mapsto \log c /2\pi i \in \ZZ/p\ZZ$\,:
  \[
    \begin{tikzcd}
      & & {H^3(\ZZ/p\ZZ,\U(1))} \arrow[dd, two heads, dashed, hook] \\ \Out^p(B)
      \arrow[rru] \arrow[rrd] & & \\ & & \{c^p =1\} \left(\xrightarrow{\e^{2\pi i x/p} \mapsfrom x }
      \ZZ/p\ZZ \right) \,.
    \end{tikzcd}
  \]
\end{example}

\begin{prop}
  \label{prop:Boeckenstein} Consider the short exact sequence of locally compact
  abelian groups
  \begin{align}
    \label{eq:ShortExactSequence}\{0\}\rightarrow \mathbb{Z}\xrightarrow{\iota}\mathbb{R}\xrightarrow p \U(1) \rightarrow \{1\}\,,
  \end{align}
  where $\iota\colon \ZZ \hookrightarrow \RR$ is the natural embedding and $p\colon
  \RR \to \U(1), t\mapsto \e^{2\pi\ima t}$. 
  Consider the \emph{Bockstein
  homomorphisms} given by
  \begin{align}\label{connector}
    \delta_{i}\colon H^{i}_{\text{Borel}}(G, \U(1)) \rightarrow H^{i+1}(G, \ZZ)
  \end{align}
  with $[b]=\delta_{1}[\id_{\U(1)}] \in H^{2}_{\text{Borel}}(G, \ZZ)$. 
  Then
  there exists a Borel section $s\colon\U(1) \rightarrow \RR$ such
  that
  \begin{align*}
    b(x,y)=s(x)+s(y)-s(x+y)\,.
  \end{align*}
\end{prop}
\begin{proof}
Let $M, N$ be any $G$-modules and $\rho \colon M \rightarrow N$ be a homomorphism, then we have associated push forward map $\rho_{\ast} \colon C^{n}(G,M) \rightarrow C^{n}(G, N)$ for all $n \in \NN$ given by 
\begin{align*}
\rho_{\ast}(f)(g_1, \ldots, g_n)=\rho(f(g_1, \ldots, g_n))
\end{align*}
for all $g_1, \ldots,g_n \in G$.
Note that if $\rho$ is a Borel map, then $\rho_{\ast}$ is a Borel map on Borel cochain complex.
Thus the short exact sequence \eqref{eq:ShortExactSequence} gives rise to the
  short exact sequence
  \begin{align}\label{chaincmplx}
    \{0\} \rightarrow C^{n}(G, \ZZ) \xrightarrow{i_{*}}C^{n}(G, \RR) \xrightarrow{p_{*}}C^{n}(G, \U(1)) \rightarrow \{1\}
  \end{align}
  on the Borel cochain complex for any abelian group $G$. Consider the identity
  map $\id_{\U(1)}\colon \U(1) \rightarrow \U(1)$ and note that $\partial \circ
  \id_{\U(1)}=0$. 
  Since $p_{\ast}$ is surjective, we have $p_{*}(b')=\id_{\U(1)}$.
  Let $s\colon \U(1) \rightarrow \RR$ be a section.
   Because of the property $p\circ
  s=\id_{\U(1)}$, we can choose $b'= s$.
   Note $\partial \circ s \in \ker(p_{*})
  =\operatorname{im}(\iota_{*})$.
   Hence there exists $a \in C^{2}(\U(1), \ZZ)$
  such that $\iota_{*}(a)(x,y)=\partial \circ s(x, y)$, which yields $a(x,y)=s(
  x)+s(y) -s(x+y)$. 
  Therefore $b(x,y)=s(x)+s(y)-s(x+y)$.
\end{proof}

\section{Computation of
\texorpdfstring{$H^{3}_{\text{Borel}}(\U(1), \U(1))$}{H3Borel(U(1), U(1))}}
\label{sec:ComputataionOfHThree}
It is more common to describe anomalies using $H^{4}
(B\U(1), \ZZ)$, e.g. see \cite{Hen17}, where $B\U(1)$ is the classifying space of $\U(1)$.
In this section, we want to identify the
cocycle from Theorem~\ref{thm:CocycleOneDimensional} with its class in $H^{4}(B\U(1), \ZZ)$. 
In particular, we show that the class associated with the cocycle
$\omega$ gives a generator of $H^{4}(B\U(1), \ZZ)$ by Theorem~\ref{thm:CocycleOneDimensional}.

Let $G$ be a polish toplogical group and $M$ be a trivial $G$-module. 
As $G$ is a polish topological group, any $\Borel$ homomorphism on $G$ is continuous, see \cite[Theorem~4, Chapter~9, Section~6]{Bo1974}.
Therefore, we have $H_{\Borel}^{1}(G,M)=\Hom(G,M)$. 
Consider the short exact sequence of locally compact
(abelian) topological groups
\begin{align*}
  \{0\} \longrightarrow \ZZ \longrightarrow \RR \xrightarrow{t\mapsto \e^{2\pi\ima t}}\U(1) \longrightarrow \{1\}\,.
\end{align*}
We get a long exact sequence with the \emph{Bockstein homomorphism}
$\delta_{i}\colon H_{\Borel}^{i}(G,\U(1))\to H_{\Borel}^{i+1}(G,\ZZ)$ for all $i \in \NN$.

In particular, we have 
\begin{align*}
H^{1}_{\Borel}(\U(1), \U(1))=\Hom(\U(1),\U(1))\cong \ZZ\,,
\end{align*}
where we have the underlying isomorphism given by
\begin{align*}
\ZZ \ni n \to \U(1) \ni z \mapsto z^{n} \in \U(1)\,.
\end{align*}
Let $[b]=\delta_{1}([\id_{\U(1)}]) \in H^{2}_{\Borel}(G,\ZZ)$. 
By Proposition
\ref{prop:Boeckenstein}, the class $[b]$ can be represented as
\begin{align*}
  b(x,y)& = s(x)+s(y)-s(x+y)&x, y \in \U(1)
\end{align*}
for a Borel section $s\colon \U(1)\to \RR$. 
For example, we can choose $s(\e^{2\pi\ima
t})= t$ for $t\in [0,1)$. 
Note that $[b]$ is a generator because the \emph{Bockstein
homomorphism} is an isomorphism by \cite{ChGuLiWe2013}.

By Theorem~\ref{thm:CocycleOneDimensional}, recall $\omega_{1}(x, y, z) = \e^{\mp 2\pi\ima s(x)(s(y)+s(z)-s(y+z))}$ for all $x, y, z \in \U(1)$ .
We claim
$\delta_{3}([\omega_{1}]) = [b]\smile [b]:= [b\wedge b]$ where
\begin{align}
  (b\wedge c)(x_{1},\ldots,x_{4})=b(x_{1},x_{2})c(x_{3},x_{4}) \text{ for all } x_1,x_2,x_3,x_4 \in \U(1)\,.
\end{align}
Namely, a straightforward computation shows that the two expressions
\begin{align}
  \delta_{3}([\omega_{1}])(x_{1},\ldots,x_{4}) & \cong s(x_{2})(s(x_{3})+s(x_{4})-s(x_{3}+x_{4}) )+                         \\
                                               & - s(x_{1}+x_{2})(s(x_{3})+s(x_{4})-s(x_{3}+x_{4}))+                        \\
                                               & + s(x_{1})(s(x_{2}+x_{3})+s(x_{4})-s(x_{2}+x_{3}+x_{4}))+                  \\
                                               & - s(x_{1})(s(x_{2})+s(x_{3}+x_{4})-s(x_{2}+x_{3}+x_{4})) +                 \\
                                               & + s(x_{1})(s(x_{2})+s(x_{3})-s(x_{2}+x_{3}))                               \\
  (b\wedge b)(x_{1},\ldots,x_{4})              & \cong (s(x_{1})+s(x_{2})-s(x_{1}+x_{2}))(s(x_{3})+s(x_{4})-s(x_{3}+x_{4}))
\end{align}
are equal for all $x_1,\cdots,x_4 \in \U(1)$.

Consequently, we have
\begin{equation}\label{4cohomeq}
  H^{4}(B\U(1), \ZZ) =\langle [b\wedge b]\rangle \,,
\end{equation}
see for example, \cite[Theorem~6]{Hen17} and
$\delta_{3}\colon H^{3}_{\Borel}(\U(1),\ZZ)\to H^{4}(B\U(1),\ZZ)$ is an
isomorphism by \cite[Remark~ 4.16]{WaWo2015}.

In particular, since $\delta_{3}([\omega_{1})] = [b\wedge b]$, we get that
$[\omega_{1}]$ generates $H^{3}_{\Borel}(\U(1), \U(1))$. 
Since we can realize $[
\omega_{1}]$ as an obstruction and the trivial obstruction can be realized by an
outer $\U(1)$ action, we have shown:
\begin{thm}
  \label{ExistenceOfObstruction} All obstructions for $\U(1)$ are realized by a
  $\U(1)$--kernel on the hyperfinite type $\mathrm{III}_{1}$ factor.
\end{thm}

\subsection{Computation of obstruction in \texorpdfstring{$n$}{n}-dimensional torus}
\label{ssec:ComputationOfobstructionNDimensional}
Let $L$ be an even positive lattice of rank $n$ and $F = L\otimes_\ZZ \RR$ 
the associated $n$-dimensional Euclidean space.
Let us consider the conformal net $\A_F$ associated with $F$ and 
fix a proper interval $I$.
We recall that given a smooth function $\ell_0\colon \SOne\to\RR$ with $q_\RR(\ell_0)=1$ and $\supp(\ell_0) \subset I$, we can consider the family of functions $\{f\cdot \ell_0\}_{f\in F}$ with
$q_F(f\cdot\ell_0) = f$, where $q_F$ is the charge on $F$.
This fixes a homomorphism $F\to\Aut(\A_F(I))\colon f \mapsto \rho_{f\cdot\ell_0}$.
By abuse of notation we denote $\rho_{f\cdot\ell_0}$ by simply $\rho_f$, for any 
$f \in F$.

Recall that by Proposition~\ref{localextn} we have a local extension $\A_L\supset\A_F$
coming from the twisted crossed product $\A_L(I) = \A_F(I)\rtimes_{\rho,b} L$.
Here  $b \colon L \times L \rightarrow \U(1)$ is a 2-cocycle given for example by
  \begin{equation*}
    b(e_{i}, e_{j}) =
    \begin{cases}
      (-1)^{\langle e_i , e_j\rangle_{F}} & \text{if}\quad i < j \\
      1                                   & \text{else}\,,
    \end{cases}
  \end{equation*}
where $\{e_1,\dots,e_n\}$ is a basis of $L$.

Denote $\iota\colon \A_F(I)\to\A_L(I)$ the canonical inclusion and 
by $\{\psi_\ell\}_{\ell\in L}$ the unitaries implementing the crossed product.
In particular, we have 
$\psi_\ell\iota(x) = \iota(\rho_f(x))\psi_\ell$ for all $x\in \A_F(I)$
and
\begin{align*}
  \psi_{\ell}\psi_{m}= b(\ell, m) \psi_{\ell+m}\text{ for all }\ell, m \in L\,. 
\end{align*}
Recall that the $\alpha$-induction on the inclusions $\A_F(I)\subset\A_L(I)$ is defined 
by
\begin{align*}
  \beta_{t}(\iota(a)\psi_m)&=\iota(\rho_t(a)\varepsilon(\rho_t, \rho_m)^\ast)\psi_{m}
  \,& t \in F\,.
\end{align*}
By following the same construction done in $\U(1)$--kernel, it follows that the map
$\lambda\colon F \times L \rightarrow \U(1)$ is given as
\begin{align*}
\lambda(t, \ell)=\iota(\varepsilon(\rho_{t},\rho_{\ell})^{\ast})\,,
\end{align*} 
where $\varepsilon$ is the associated braiding in the Equation~\eqref{braidtu} on $\A_F(I)$.
Choose a Borel section
$s \colon F/L \rightarrow F$. 
Following \cite{Jo1980} for all $q, r, t \in F/L$, the obstruction $\omega_{\mp}$ is given by
\begin{align*}
  \omega_{\mp}(q,r,t) 
    ={} & b(s(q)+s(r)-s(q+r), s(q+r)+s(t)-s(q+r+t)) \times \\
        & b(s(r)+s(t)-s(r+t), s(q)+s(r+t)-s(q+r+t))^{\ast}\times \\
        & \e^{\mp\pi\ima\langle s(q), s(r)+s(t)-s(r+t) \rangle} \\
   ={}  & \e^{\mp\pi\ima ( [s(q)+s(r)-s(q+r), s(q+r)+s(t)-s(q+r+t)] 
          - [s(r)+s(t)-s(r+t), s(q)+s(r+t)-s(q+r+t)]) } \times\\
        & \e^{\mp\pi\ima  \langle s(q), s(r)+s(t)-s(r+t) \rangle }
\end{align*}
where
\begin{align}
  [x,y]+[y,x] & = \langle x,y\rangle \pmod 2\quad\text{and}\quad b(x,y)=(-1)^{[x,y]}
\end{align}
and
\begin{align*}
  s(f) = \sum_{i=1}^{n} s(f_{i})\cdot e_{i} \quad\text{for}\qquad f =\sum_{i=1}^{n} f_{i}\cdot e_{i}\,.
\end{align*}
Here $x, y \in F$.
Namely, we have the following ($s_{i_1\cdots i_n}= s(i_{1}+\cdots + i_{n})$ $s
_{i_1\cdots i_n|j_1\cdots j_m}= s(i_{1}+\cdots + i_{n}) + s(j_{1}+\cdots +j_{m}
) - s(i_{1}+\cdots +i_{n}+j_{1}+\cdots j_{m})$)

\begin{align}
\label{eq:comptcohom}
  2\delta_{3}([\omega_{L}])(x_{1},\ldots,x_{4}) 
    & \cong \langle s_{2}|s_{3|4}\rangle + [s_{2|3}|s_{23|4}] - [s_{3|4}|s_{2|34}]+{}    \\
    & - \langle s_{12}| s_{3|4}\rangle - [s_{12|3}|s_{123|4}] + [s_{3|4}|s_{12|34}] +{}  \\
    & + \langle s_{1}| s_{23|4}\rangle +[s_{1|23}| s_{123|4}] - [s_{23|4}|s_{1|234}]+{}  \\
    & - \langle s_{1}| s_{2|34}\rangle - [s_{1|2}| s_{12|34}] + [s_{2|34}| s_{1|234}]+{} \\
    & + \langle s_{1}| s_{2|3}\rangle + [s_{1|2}|s_{12|3}] - [s_{2|3}| s_{1|23}]\,. 
\end{align}
\begin{rmk}
  \label{rmk:omegaG} Let us compare this with cocycles on $\U(1)^{n}$. 
Choosing a fixed basis $(e_{1},\ldots,e_{n})$ of $L$ gives us an identification
  $\U(1)^{n} \to F/L$ via the isomorphism
  \begin{align}\label{x_i}
    \left(\e^{2\pi \ima x_i}\right)_{i=1,\ldots n}                      & \mapsto\sum_{i} x_{i} e_{i} +L \qquad x\in \RR^{n} ,~ x_i \in \RR \\
    \left(\e^{2\pi\ima \langle e_i^\ast|x\rangle}\right)_{i=1,\ldots,n} & \mapsfrom x+L\,,
  \end{align}
  where $(e^{\ast}_{1},\ldots, e^{\ast}_{n})$ is the dual basis of
  $(e_{1},\ldots, e_{n})$ with respect to the inner product
  $\langle\slot , \slot\rangle$, \ie $\langle e_{i}^{\ast},e_{j}\rangle = \delta
  _{i,j}$. 
  Under this identification, we can see $\omega_{L}$ as a class $[\omega
  _{G}]=H^{3}_{\Borel}(\U(1)^{n},\U(1))$ only depending on the Gram matrix $G_{i,j}
  =\langle e_{i} , e_{j}\rangle_{F}$. 
  Note that the pull back of the scalar
  product $\langle\slot , \slot\rangle_{F}$ under the map
  $\RR^{n} \to F, x\mapsto \sum_{i=1}^{n} x_{i}e_{i}$ is given by
  $\langle\slot, \slot\rangle_{G}:=\langle\slot, G\slot\rangle_{\RR^n}$, where $\langle
  \slot, \slot\rangle_{\RR^n}$ is the standard Euclidean inner product on $\RR^{n}$.
  Similarly, we can choose $[\slot|\slot]_{G} =\langle \slot, G^{\utr} \slot\rangle
  _{\RR^n}$ where $G^{\utr}$ is the upper triangle part of $G$.
   Furthermore, consider
  \begin{align}\label{bmap}
    b_{i}(x_{1},x_{2}) = s_{1|2,i}:= s(x_{1,i}) + s(x_{2,i}) - s(x_{1,i}+x_{2,i})
  \end{align}
  as our standard generators of $H^{2}(\U(1)^{n},\ZZ) = \langle b_{1},\ldots,b_{n}
  \rangle$.
\end{rmk}
\begin{example}
  \label{ex:A2} For the $A_{2}$ root lattice, we have
  \begin{align}
    G =\begin{bmatrix}2&-1\\-1&2\end{bmatrix} \qquad G^{\utr}:= \begin{bmatrix}0&-1\\0&0\end{bmatrix}
  \end{align}
  and for any $x, y \in \RR^n$, we define
  \begin{align}
    \langle x | y\rangle_{G} & = \langle x|Gy\rangle_{\RR^n}                      \\
    [x|y]_{G}                & = -x_{1}y_{2} = \langle x|G^{\utr}y\rangle_{\RR^n}
  \end{align}
  A lengthy but straightforward calculation shows
  \begin{align}
    -2[\delta_{3}([\omega_{G}])-b_{1}\wedge b_{1} - b_{2}\wedge b_{2}](x_{1},\ldots,x_{4}) & = s_{2,1}s_{3|4,2}+ s_{2|3,1}s_{23|4,2}- s_{3|4,1}s_{2|34,2}     \\
                                                                                           & - s_{12,1}s_{3|4,2}- s_{12|3,1}s_{123|4,2}+ s_{3|4,1}s_{12|34,2} \\
                                                                                           & + s_{1,1}s_{23|4,2}+s_{1|23,1}s_{123|4,2}- s_{23|4,1}s_{1|234,2} \\
                                                                                           & - s_{1,1}s_{2|34,2}- s_{1|2,1}s_{12|34,2}+ s_{2|34,1}s_{1|234,2} \\
                                                                                           & + s_{1,1}s_{2|3,2}+ s_{1|2,1}s_{12|3,2}- s_{2|3,1}s_{1|23,2}     \\
                                                                                           & + s_{3|4,1}s_{2,2}- s_{3|4,1}s_{12,2}+ s_{23|4,1}s_{1,2}         \\
                                                                                           & - s_{2|34,1}s_{1,2}+ s_{2|3,1}s_{1,2}                            \\
                                                                                           & =2s_{1|2,1}s_{3|4,2}=2(b_{1}\wedge b_{2})(x_{1},\ldots,x_{4})
  \end{align}
  where in the first step the diagonal terms cancel as in the $\U(1)$-case and
  the second step follows after expanding.
 In particular, we have
  \begin{align*}
    \delta_{3}([\omega_{G}]) & = (b_{1}\wedge b_{1}+b_{2}\wedge b_{2}) - (b_{1}\wedge b_{2}) \,.
    \end{align*}
\end{example}
\begin{prop}
  \label{rootlaticecoho} Fix a basis $\{e_1, \ldots, e_n\}$ of an even positive lattice $L \subseteq F$, where $F$ is an $n$-dimensional Euclidean space.
  Let $G=(G_{i,j}:1\leq i, j \leq n\}$ be the Gram matrix with respect to the basis and $b_i$ be defined in the Equation~\eqref{bmap}.
  Then using the notation of Remark
  \ref{rmk:omegaG}, we have
  \begin{equation}
    \delta_{3}([\omega_{G}]) = \frac{1}{2} \sum_{i} G_{i,i}\cdot b_{i}\wedge b_{i}
    +\sum_{i< j}G_{i,j}\cdot b_{i}\wedge b_{j} 
    \,, \qquad G_{i,j}=\langle e_i, e_j \rangle_F\,.
  \end{equation}
\end{prop}
\begin{proof}
Let $\{E_{i, j}: 1\leq i, j \leq n\}$ be the standrad matrix units of $M_n(F)$.
  This follows by writing $G = \sum_{i} G_{i,i}E_{i,i}+\sum_{i<j}G_{i,j}(E_{i,j}
  + E_{j,i})$ and we already have checked it for $2E_{i,i}$ (by Equation~\eqref{eq:comptcohom}) and
  $2E_{i,i}+2E_{j,j}- E_{i,j}-E_{j,i}$ in Example~\ref{ex:A2}. Thus, the
  result follows in general by linearity.
\end{proof}
\begin{rmk}
Note that $G$ needs to be the Gram matrix of a symmetric even bilinear
form on $\ZZ^{n}$. 
It follows immediately that
\begin{equation}\label{cohompd}
  \omega_{G_1}\omega_{G_2}= \omega_{G_1+G_2}\,, G_1, G_2 \text{ are Gram matrices}\,.
\end{equation}

Choosing a basis in
$\Hom(F/L,\TT) = \{\e^{2\pi \ima \langle x|\slot\rangle}: x \in L^{\ast}\} \cong
L^{\ast}$ (dual lattice of $L$),
we define
\begin{equation}
  \tilde b_{i}(x_{1},x_{2}) = \langle e_{i}^{\ast} | s_{1} + s_{2} - s_{12}\rangle
\end{equation}
and get that $\langle \tilde b_{i}: i=1,\ldots,n \rangle = H^{2}(F/L,\ZZ)$.
It follows that
\begin{equation}
  \sum_{i,j}\frac{1}{2} \langle x| e_{i}^{\ast}\rangle G_{i,j}\langle e_{i}^{\ast}
  |y \rangle = \frac{1}{2} \langle x|y\rangle\,, \qquad x, y \in L^{\ast}\,.
\end{equation}
Let $\beta \colon L\times L \to \ZZ$ be a bicharacter with
$\beta(x,y)+\beta(y,x) = \langle x | y\rangle$ for all $x,y \in L$. Then the class
$[d_{L}] \in H^{4}(T_{L},\ZZ)$ can be represented as
\begin{equation}
  d_{L}(x_{1},x_{2},x_{3},x_{4}) = \beta(s_{L}(x_{1}) +s_{L}(x_{2}) - s_{L}(x_{1}
  +x_{2}), s_{L}(x_{1})+s_{L}(x_{2})-s_{L}(x_{3}+x_{4}))
\end{equation}
where $s_{L}\colon T_{L} \to F$ is a section for the short exact sequence
\[
  \begin{tikzcd}
    \{0\} \arrow[r] & L \arrow[r] & F \arrow[r] & T_L \arrow[r] \arrow[l, "s_L"',
    dashed, bend right] & \{0\}\,.
  \end{tikzcd}
\]
\end{rmk}
\begin{cor}
  \label{cor:MainResult}
  Let $R$ be the hyperfinite type $\mathrm{III}_1$ factor and $n\in \NN$. 
  Let $T$ be
  a compact $n$-torus.
  Then all obstructions in $H^{3}_{\Borel}(T,\U(1))$ arise
  from a $T$-kernel on $R$.
  In other words, there are anomalous actions of $T$
  on $R$ for any class in $H^{3}_{\Borel}(T,\U(1))$.
\end{cor}
\begin{proof}
  Without loss of generality, let $T=\U(1)^{n}$.
  Let $[\omega] \in H^{3}_{\Borel}(T, \U(1))$.
  From the Künneth formula it follows that  
  \begin{align*}
  [\omega]=\sum_{1\leq i\leq j\leq n}G_{ij}\delta_3([b_i] \cup [b_j])\,.
  \end{align*}
  For $G \in M_n(\ZZ)$, assume $G_{ij}=0$ for $j<i$.
Consider the symmetric matrix $H=G+G^{T} \in M_n(\ZZ)$ with even values in diagonal.  
If $H$ is the Gram matrix of an even positive lattice of rank $n$, the class $[\omega]$ is realized in $\A_L$ by Proposition~\ref{rootlaticecoho} .

For general $H$, note that it suffices to show that 
$H$ is in the abelian group generated by Gram matrices of positive even lattices by Equation~\eqref{cohompd}.
Note that diagonal lattices $\sqrt{2m_1}\ZZ \oplus \dots \oplus \sqrt{2m_n}\ZZ$ with $(m_1,\dots,m_n) \in \ZZ^n$ generate all diagonal even matrices.

Note that $A_2 \oplus A_1 \oplus \dots \oplus A_1$ has Gram matrix $\diag(2,\dots,2)-E_{2,1}-E_{1,2}$ for $1 \leq i <j \leq n$, where $E_{i,j}$ is standard matrix unit.
Taking index permutations we get $\diag(2,\dots,2)-E_{i,j}-E_{j,i}$ for $1 \leq i<j\leq n$.
Together with the diagonal even matrices this generates the abelian group of symmetric integer matrices with even diagonal, which concludes the proof.
\end{proof}

\section{Local extensions}
\label{ssec:localextensions}
Let $I, J \subset S^1 \setminus \{-1\}$ be open proper intervals and 
$\A=\{\A_F(I): I \in \cI\}$ and let $\Delta = \{\rho_\ell : \ell \in L\}$.
Recall $\rho$ is defined in Section ~\ref{ssec:ComputationOfobstructionNDimensional}.  
 Define $\langle\iota(\A_F(I)),\{\psi_{\rho}\}_{\rho\in\Delta}
  \rangle:=\A_L(I)$ with $\rho$ localized endomorphism of $\A$ localized in $I$

  and $\psi_{\rho}$ charged fields, \ie
  \begin{align}\label{locl}
    \psi_{\rho}\iota(a) & = \iota\rho(a)\psi_{\rho} & a\in \A_F(I)\,.
  \end{align}
  Define
  \[
     \langle\iota(\A_F(J),\{\iota(z_{\rho,J})\psi_{\rho}
    \}_{\rho}\rangle:=\A_L(J)
  \]
  where $z_{\rho,J}$ (charge transporter) is a unitary such that $\rho^{J}=\Ad z_{\rho,J}\circ\rho$
  is localized in $J$.
  It follows that $z_{\rho,J}$ is localized.
\begin{prop}
\label{prop:diffreg}
  We show $\A_L$ is local if and only if
  \begin{align}\label{crossbraid}
    \psi_{\rho}\psi_{\sigma} & = \iota(\varepsilon(\sigma,\rho)) \psi_{\sigma}\psi_{\rho} & \text{for all }\rho,\sigma\in \Delta \,.
  \end{align}
\end{prop}

\begin{proof}
  This follows from the following computation.
   Let $I_{1},I_{2}$ be disjoint intervals.
  Then
  \begin{align}
    \iota(z_{\rho,I_1})\psi_{\rho} \iota(z_{\sigma,I_2})\psi_{\sigma} & =\iota(z_{\rho,I_1}\rho(z_{\rho,I_2})) \psi_{\rho}\psi_{\sigma}   \\
    \iota(z_{\sigma,I_2})\psi_{\sigma}\iota(z_{\rho,I_1})\psi_{\rho}  & =\iota(z_{\sigma,I_2}\sigma(z_{\rho,I_1})\psi_{\sigma}\psi_{\rho}
  \end{align}
  which are equal if and only if
  \begin{align}
    \psi_{\rho}\psi_{\sigma} & = \iota(\rho(z_{\rho,I_2}^{\ast}) z_{\rho,I_1}^{\ast} z_{\sigma,I_2}\sigma(z_{\rho,I_1})\psi_{\sigma}\psi_{\rho} =\iota(\varepsilon^{\pm}(\sigma,\rho))\psi_{\sigma}\psi_{\rho}
  \end{align}
  where $\pm$ depends on the order of $I_{1}$ and $I_{2}$.
  In other words, locality of the ambient extension directly follows from \cite[Theorem~6.8]{DeGi2018}.
\end{proof}

\begin{example}
Let
  \[
    \left\{(x_{1},\ldots,x_{n+1}) \in \ZZ^{n+1}: \sum_{i=1}^{n+1}x_{i}= 0\right
    \} \subseteq \left\{x\in \RR^{n+1}: \sum_{i=1}^{n+1}x_{i}=0\right\} \cong \RR
    ^{n}\,.
  \]
  Then there is a basis $e_{i}=(\delta_{i,j}- \delta_{i+1,j})_{j} = (0,\ldots,
  0,1,-1,0,\ldots,0)$ with the Gram matrix
  \begin{align*}
    \langle e_{i},e_{j}\rangle =\begin{cases}2&\text{if }i=j\\ -1&\text{if }|i-j|=1\\ 0&\text{if }|i-j|\leq 2\,.\end{cases}
  \end{align*}
  Let the construction and hypothesis be the same as in Proposition ~\ref{localextn}.
\end{example}
So we must need that
$\varepsilon(\rho_{m},\rho_{m}) = \e^{\pi\ima \langle q_m,q_m\rangle}\stackrel{!}{=}
1$.
In particular, we have
\[
  \varepsilon(\rho_{m},\rho_{n}) = \e^{\pi \ima \langle q_m,q_n\rangle}=(-1)^{\langle
  q_m,q_n\rangle}\,.
\]
It is clear that $\langle x,x\rangle \in 2\ZZ$ for all $x\in L$, which implies
that $\langle x,y \rangle \in \ZZ$ for all $x,y\in L$. 
Given a lattice with the
basis $\{e_{1},\ldots, e_{n}\}$, choose
$b\colon L \times L \to \{\pm 1\}\leq \U(1)$ such that
\begin{align*}
  b(e_{i},e_{j}) = \begin{cases}(-1)^{\langle e_i,e_j\rangle}&\text{if }i\leq j\\ 0&\text{otherwise}\,.\end{cases}
\end{align*}
In particular, one has
$b(e_{i},e_{j})b(e_{j},e_{i})^{\ast} = (-1)^{\langle e_i,e_j\rangle}$.
Therefore, one obtains
\[
  \psi_{i}\psi_{j} = b(e_{i},e_{j})\psi_{i+j}= b(e_{i},e_{j})b(e_{j},e_{i})^{\ast}
  \psi_{j}\psi_{i} = (-1)^{\langle e_i,e_j\rangle}\psi_{j}\psi_{i}\,,
\]
where $\psi_{i}$ are the unitaries which implement the crossed product $\A_{\RR^n}(I)\rtimes_{b} L$ .
Thus
\[
  \A_{\RR^n} \subseteq \A_{\RR^n}\rtimes_{b} L
\]
is a local extension.

\subsection{Topology on sectors}
\label{ssec:Topology}
Let $\A$ be a local $\Mob$ covariant conformal net on a Hilbert space $\Hil$ (see Definition~\ref{cft}).
We define an automorphism of the net $\A$
written as $\alpha \in \Aut(\A)$ to be a family $\alpha=\{\alpha_{I}\in\Aut(\A(I))\}$ which is compatible under inclusion, ${\alpha_I}|J = \alpha_{J}$ for $J\subset I$ and $\Omega$-preserving, \ie $(\Omega,\alpha_{I}(x)\Omega)=(\Omega, x\Omega)$ for all $x \in \A(I)$.
We say $U\in \U(\A)$ if the $U\in \U(\Hil)$ subject to $U\Omega = \Omega$ and $\Ad U \in \Aut(\A(I))$ for all $I\in\cI$.
Note that there is a one-to-one correspondence between $\U(\A)$ and $\Aut(\A)$. 
Namely, given $U \in \U(\A)$, define $\alpha_{I} =\Ad U$ and given $\alpha$, define
$Ux\Omega = \alpha_{I}(x)\Omega$ for some interval $I$. 
Note that such $U$ does not
depend on $I$.
Let $G\leq \Aut(\A)$ be a subgroup which is compact\footnote{Let $\A$ be a
conformal net with split property. 
Then $\Aut(\A)$ is a compact metrizable group
\cite{DoLo1984}. 
Thus every closed subgroup is compact.}. 
Let $I_{0}\in\cI$ with $I_{0}\Subset S^{1}\setminus\{-1\}$. 
We say $\rho \in \tRep G^{I_0}(\A)$ if $\rho=\{\rho_{I}\colon \A(I)\to \cB(\Hil)\}$
is a compatible family of normal $\ast$-morphisms and there is an automorphism
$\alpha \in G$ such that
\[
  \rho_{I}(a) =
  \begin{cases}
    a             & \text{if }a\in \A(I) \text{ and }I< I_{0} \\
    \alpha(a) & \text{if }a\in \A(I)\text{ and }I > I_{0}\,.
  \end{cases}
\]
Define $\partial \rho = \alpha$.

Let $I_{0}\in\cI$ be an interval. 
We want to equip
$\tilde G = \pi_{0}(\tRep G^{I_0}(\A)^{\times})$ with a topology. 
Here for $\cC$ is
a tensor category, where $\cC^{\times}$ denotes 2-group which is the full and replete
subcategory of invertible objects and $\pi_{0}(\cC)$ the isomorphism classes of
objects. 
Note that we have a natural embedding
$\pi_{0} \colon \tRep G^{I_0}(\A)^{\times}\to\Out(\A(I_{0}))$ by
$[\rho]\mapsto [\rho_{I_0}]$. 
Choose $I_{1},I_{2}\in\cI$ with $I_{2} \Supset I_{1}
\Supset I_{0}$. (Here $I \Supset J$ means $I \supset \bar{J}$ with $I, J \in \cI$)

By \cite{Mg2005}, note that $\tRep G^{I_0}(\A)$ is the category of $G$-twisted representations of $\A$ that are $G$-localized in $I_0$.

\textbf{Open question:}
Let $G\leq \Aut(\A)$ be a compact group. 
Consider $\pi_0(\tRep G^{I_0}(\A^{\times}))$, the group formed under composition of equivalence classes of $G$-twisted invertible representations of $\A$ which are $G$-localized in $I_0$ (\cite{Mg2005}).
Call $H=\pi_0(\tRep G^{I_0}(\A^{\times}))$. 
  The map
  \[
    \partial \colon \pi_{0}(\tRep G^{I_0}(\A)^{\times}) \to G,~ [\beta]\mapsto
    \partial\beta
  \]
 is  a twisted representation to its degree.
We equip $\pi_{0}(\tRep G^{I_0}(\A)^{\times})$ with the initial topology given by the degree map $\partial \colon H \rightarrow G$.
When is $H$ compact?
What are the sectors associated to it and what is $H$--kernel?
Is there an intrinsic topology on $G$--twisted sectors (see for \eg \cite{Mg2005} )?
\begin{conj}
  If $\A$ is completely rational and $G\leq \Aut(\A)$ closed, then
  $\pi_{0}(\tRep G^{I_0}(\A)^{\times})$ is compact. Indeed, let $G_{0}\leq G$
  be the image of $\partial$, then we have a short exact sequence
  \[
    1\to \pi_{0}(\Rep^{I_0}(\A)^{\times})\to \pi_{0}(\tRep G^{I_0}(\A)^{\times}
    ) \to G_{0}\to 1
  \]
\end{conj}
\begin{rmk}
In the Example~\ref{exmpltwisted}, for an even lattice $L \subset F$ with $F$ an $n$-dimensional Euclidean space,
we have the short exact sequence 
\[
  0 \to L^{\ast}/L \to F/L \to F/L^{\ast} \to 0\,.
\]
Hence the map $\partial \colon H=\pi_0(\tRep G^{I_0}(\A^{\times})):=F/L \rightarrow G:=F/L^{\ast}$ is a finite covering map. 
So the topology is well-defined.
\end{rmk}
\subsection{\texorpdfstring{$F/L$}{F/L}--kernel is \texorpdfstring{$F/L^{\ast}$}{F/L*}--twisted}
\label{ssec:twistedkernels}
\begin{example}\label{exmpltwisted}
Let $L^{\ast} = \{x\in F: \langle x,L \rangle \subset \ZZ\}$ be the dual
lattice and $L\subset L^{\ast}$.
  If $F=\RR$ and $L=\sqrt{2n}\ZZ$, then $L^{\ast} = \frac{1}{\sqrt{2n}}\ZZ$
  and $L^{\ast}/L \cong \ZZ/2n\ZZ$.
\end{example}
Let $(F, \langle \slot,\slot\rangle)$ be a $n$-dimensional Euclidean space and $L \subseteq F$ be an even lattice.
We start with $\A_{F}$ and we choose a reference interval $I_{0}$ and $\{\rho_{q}\}_{q\in
L}$ localized in $I_{0}$.
Then, define that $L\to \Aut(\A_{F}(I))$ a homomorphism.
The local net associated to the lattice $L$ is defined by
\[
  \A_{L}(I_{0}) := \A_{F}(I_{0})\rtimes_{b}L = \langle \iota(\A_{F}(I_{0})) ,\psi
  _{q} : q\in L \rangle
\]
where $\psi_{q} \in (\iota,\iota\rho_{q})$ and
\[
  \A_{L}(I) := \langle \iota(\A_{F}(I)), \iota(z_{q,I_0,I})\psi_{q} : q\in L\rangle
\]
with unitaries $z_{q,I_0,I}\in (\rho_{q},\rho^{I}_{q})$ and $\rho^{I}_{q}$ is
localized in $I$,
where $b \colon L\times L \longrightarrow \ZZ_2$ is a bilinear form as in \cite{Ka1998} fulfilling
\begin{align*}
    b(\alpha,\alpha) = \frac 12 \langle \alpha,\alpha\rangle 
    \quad \text{for all } \alpha \in L
    \komma
\end{align*}
\eg if $\{\alpha_1,\ldots,\alpha_n\}$ is a basis of $L$, then we can choose
\begin{align*}
    b(\alpha_i,\alpha_j) &= 
        \begin{cases} 
            \langle \alpha_i,\alpha_j \rangle \mod 2 & i<j\\
            \frac 12 \langle \alpha_i,\alpha_i\rangle \mod 2 & i=j\\
            0 &i>j 
            \punkt
        \end{cases}
\end{align*}
On $A\rtimes_{b} L$ we have the so-called global action (inner symmetry) of the
torus $F/L^{\ast}$ densely defined by linearly extending
\begin{align}\label{globalact}
  \alpha_{f+L^\ast}(\iota(a)\psi_{x}) & = \e^{-2\pi \ima \langle f,x\rangle}\iota(a)\psi_{x} & x\in L
\end{align}
and $(A\rtimes_{b} L)^{F/L^\ast}= A$. 
Note that
$F/L^{\ast} \ni f+L^{\ast} \to (x\to\e^{2\pi \ima \langle f,x\rangle}
) \in \hat{L}$
is the natural isomorphism identifying $F/L^{\ast}$ with the Pontryagin
dual $\hat L = \Hom(L,\U(1))$.
Notice that for any lattice $\Gamma \subset F$, 
  \begin{align*}
    \Gamma \mapsto \hat\Gamma=\Hom(\Gamma,\U(1))\,.
  \end{align*}
  In Pontrayagin duality, we equip $\hat\Gamma$ with a uniform topology on compact sets.
 Note that this is a compact group.
Consider the short exact sequence
\[
  0 \to L^{\ast}/L \to F/L \to F/L^{\ast} \to 0\,.
\]

We have a $F/L$--kernel acting on $\A_{L}(I_{0})$. 
Let $s \colon F/L \rightarrow F$ be a Borel section.
If we want to extend it to
slightly bigger interval $I = I_{-}\sqcup \bar I_{0}\sqcup I_{+}$ with
$I_{-}<I_{+}$, then we can extend it as follows to
$\A_L(I_{-})\cup\A_L(I_{+})\cup \A_L(I_{0})$ by
\[
  \tilde\rho_{s(f+L)}(a) =
  \begin{cases}
    \tilde\rho_{s(f+L)}(a) & \text{if }a\in \A_L(I_{0}) \\
    a                      & \text{if }a\in \A_L(I_{-}) \\
    \alpha_{f+L^\ast}(a)   & \text{if }a\in \A_L(I_{+})\,,
  \end{cases}
\]
which extends by strong additvity to $\A_L(I)$.
Here $\tilde{\rho}$ is extension of $\rho$ from $\A_F(I_0)$ to $\A_L(I_0)$ via $\alpha$-induction.

\begin{prop}
  \label{prop:Mueger} The automorphism $\tilde\rho_{s(f+L)}$ defined above extends normally to $\A_{L}(I)$. 
  In particular, the automorphism $\tilde
  \rho_{s(f+L)}$ is a $\alpha_{f+L^\ast}$-localized in $I_{0}$ in the sense of M\"uger
 and localized if $f \in L^{\ast}$.
\end{prop}
\begin{proof}
To prove this, we do the construction for $I$ instead of $I_{0}$ with $\ell_{0}$ having charge $1$ supported in $I_{0}$ and it suffices to check that restricted on $I_{\pm}$ the formula holds.
  First note that $I\supset I_{0}$ then
  $\A_{L}(I) = \langle \iota(\A_{F}(I)), \psi_{q}\rangle$. 
 If we use the same formula for $\tilde\rho_{f,I}$ as for $\A_{L}(I_{0})$, then trivially $\tilde
  \rho_{f,I}\restriction_{\A_L(I_0)}=\tilde\rho_{f,I_0}$.
  Here we write $s(f+L)$ as $f$ by a slight abuse of notation with a Borel section map $s \colon F/L \longrightarrow F$ for $f \in F$.

  Let us take $x = \iota(a W([q(L_{-}-L_{0})]) \psi_{q}$ where $a\in \A_{F}(I_{-}
  )$ and therefore $x\in \A_{L}(I_{-})$ where $L_{\bullet}'=\ell_{\bullet}$ and
  $\rho_{q}(W([g])) = \e^{\ima q\int\langle \ell_0, f\rangle }W([g])$ where $\tilde\rho\iota=\iota\rho$ and $\supp{L_{\bullet}'}\subset I_{\bullet}$.
  
 Denote by $\alpha_{\rho_{\ell_{0}f}, I}$ the $\alpha$-induction to $\A_{L}(I)$ corresponding to the localized action $\rho_{\ell_{0}f, I}$ on $\A_F(I)$.
 By definition, we have $\tilde{\rho}_{f,I}=\alpha_{\rho_{\ell_{0}f}, I}$.
 From Theorem~\ref{braidndim}, it follows that
 \begin{align*}
 \alpha_{\rho_{\ell_{0}f}, I}(\psi_q)=e^{- \ima \pi \langle q, f \rangle} \psi_q\,.
 \end{align*}
 Further, note that $\rho_{\ell_{0}f,I}(a)=a$ since $a \in \A_{F}(I_{-})$ and $\supp(\ell_0) \subseteq I_0$.
  Then one has
  \begin{align}
    \tilde \rho_{f,I}(x) & = \iota(\rho_{\ell_{0}f,I}(a W([q(L_{-}-L_{0})]) \tilde\rho_{f,I}(\psi_{q})                                                            \\
                         & = \iota\rho_{\ell_{0}f,I}(a W([q(L_{-}-L_{0})])\alpha_{\rho_{\ell_0f},I}(\psi_{q})                                                        \\
                           & = \iota(\rho_{\ell_{0}f,I}(a))\iota\rho_{\ell_{0}f,I}(W([q(L_{-}-L_{0})]) \e^{-\pi\ima \langle q|f\rangle}\psi_{q}        \\
                         & = \iota(a)\iota\rho_{\ell_{0}f,I}(W([q(L_{-}-L_{0})]) \e^{-\pi\ima \langle q|f\rangle}\psi_{q}                                         \\
                         & = \iota(a)\e^{\ima \int \ell_0(L_--L_0) \langle q|f\rangle }\iota(W([q(L_{-}-L_{0})]) \e^{-\pi\ima \langle q|f\rangle}\psi_{q} \\
                         & = \iota(a)\e^{\ima \pi \langle q|f\rangle }\iota(W([q(L_{-}-L_{0})]) \e^{-\pi\ima \langle q|f\rangle}\psi_{q}                  \\
                         & = \iota(a)\iota(W([q(L_{-}-L_{0})]) \psi_{q}                                                                                   \\
                         & =x\punkt
  \end{align}
  Write $y=\iota(bW([q(L_{+}-L_{0})])\psi_q$. 
  Since $b \in \A_{F}(I_{+})$ and
 $\rho_{\ell_{0}f,I}$ is localized in $I_0$, we have $\rho_{\ell_{0}f,I}(b)=b$.
   Similarly, one gets
  \begin{align}
    \tilde \rho_{f,I}(x) & = \iota(\rho_{\ell_0 f,I}(b W([q(L_{+}-L_{0})]) \tilde\rho_{f,I}(\psi_{q})                                                            \\
                         & = \iota\rho_{\ell_0 f,I}(b W([q(L_{+}-L_{0}])\alpha_{\rho_{\ell_0f},I}(\psi_{q})                                                        \\
                         & = \iota(b)\iota\rho_{\ell_0 f,I}(W([q(L_{+}-L_{0})]) \e^{-\pi\ima \langle q|f\rangle}\psi_{q}                                         \\
                         & = \iota(b)\e^{\ima \int \ell_0(L_+-L_0) \langle q|f\rangle }\iota(W([q(L_{+}-L_{0})]) \e^{-\pi\ima \langle q|f\rangle}\psi_{q} \\
                         & = \iota(b)\e^{-\ima \pi \langle q|f\rangle }\iota(W([q(L_{+}-L_{0})]) \e^{-\pi\ima \langle q|f\rangle}\psi_{q}                 \\
                    &= \e^{-2 \ima \pi \langle q|f\rangle }\iota(b)\iota(W([q(L_{+}-L_{0})]) \psi_{q}                \\
                         & = \alpha_{f+L^\ast}(x)\punkt
  \end{align}
\end{proof}
\begin{cor}
  Any irreducible localized representation of $\A_{L}$ is equivalent to some $\tilde
  \rho_{s(f+L)}$ for some $f\in L^{\ast}$.
In particular, we have an $L^{\ast}
  /L$-kernel of localized automorphisms.
\end{cor}

\begin{appendix}
  \label{sec:Appendix}
\section{Extensions}
\subsection{Extensions of endomorphisms of von Neumann algebra}
\label{ssec:ExtensionsofendomorphismsofvonNeumann algebra} Let $M$ be any
factor equipped with a faithful normal state $\varphi$ and in standard form with a cyclic and separating
vector $\Omega_{\varphi}$ acting on the GNS Hilbert space $L^{2}(M, \varphi)$.
Let $\alpha$ be a
faithful action of $\mathbb{Z}$ on $M$, \ie $\alpha \colon \ZZ \rightarrow \Aut(M)$ a monomorphism. 
Consider $M \rtimes_{\alpha}\mathbb{Z}= (
x, U_{n}) \subseteq \mathcal{B}(L^{2}(\mathbb{Z}, L^{2}(M, \varphi))$, where $x \in M$ and $U_n \in \U(L^{2}(\mathbb{Z}, L^{2}(M, \varphi)), n \in \ZZ$. 
Moreover, the extended state $\tilde{\varphi}$ associated to $\varphi$ on the crossed product $M \rtimes_{\alpha} \ZZ$ is given by
\begin{align*}
\tilde{\varphi}\left(\sum_{n \in \mathbb{Z}}a_{n} U_{n}\right)=\varphi(a_{0})\,.
\end{align*}
 Also denote by $\|\cdot
\|_{2, \varphi}$ the associated Hilbert space norm of elements in $L^{2}(M, \varphi)$ and $\|\cdot\|_{2, \tilde{\varphi}}$  the associated Hilbert space norm for elements in $L^{2}(M \rtimes_{\alpha}\ZZ, \tilde{\varphi})$.
Let $\beta \in \Aut(M)$ and $z_{n} \in \U(M)$ for all $n \in \ZZ$.
Then define
$\tilde{\beta} \colon M \rtimes_{\alpha}\ZZ\rightarrow M \rtimes_{\alpha}\ZZ$
given by
\begin{align}\label{DEFAUT}
\tilde{\beta}\left(\sum_{n \in \ZZ}a_{n} U_{n}\right)=\sum_{n \in \ZZ}\beta(a_{n})z
_{n}U_{n}\punkt
\end{align}
\begin{lem}\label{WED}
  The map $\tilde{\beta} \colon M \rtimes_{\alpha}\ZZ\longrightarrow M \rtimes_{\alpha}\ZZ$ defined in the Equation~\eqref{DEFAUT} is a well-defined map.
\end{lem}
\begin{proof}
  Let $x=\sum_{n \in \ZZ}x_{n} U_{n}$, $x_n \in M$.
  Then note that
  \begin{align}\label{crosde}
    \|x\|_{2, \tilde{\varphi}}^{2}= \sum_{n \in \ZZ}\|x_{n}\|_{2, \varphi}^{2}\punkt
  \end{align}
  It follows from the Equation~\eqref{crosde} that if $\sum_{n \in \ZZ} x_nU_n=0$, then $x_n=0$.
  If
  \begin{align*}
      x & = \sum_{n \in \ZZ}a_{n} U_{n}= \sum_{n \in \ZZ}b_{n} U_{n}=y &  a_n,~ b_n \in M,~ n \in \ZZ\,,
  \end{align*}
  then $a_{n}=b_{n}$. 
  So, one has
  \begin{align*}
    \sum_{ n \in \ZZ}\beta(a_{n})z_{n} U_{n}= \sum_{ n \in \ZZ}\beta(b_{n})z_{n} U_{n}\punkt
  \end{align*}
  Hence $\tilde{\beta}(x)=\tilde{\beta}(y)$. 
  Therefore $\tilde{\beta}$ is well-defined.
\end{proof}
\begin{thm}
Let $\tilde{\beta}$ be the map defined in the Equation~\eqref{DEFAUT}.
Then the following are equivalent
\begin{enumerate}
\item   $\tilde{\beta}\in \Aut(M \rtimes_{\alpha}\ZZ)$,
\item $z_{n+s}=z_{n}\alpha_{n}(z_{s})$
  and
  $\beta(\alpha_{n}(x))z_{n}=z_{n} \alpha_{n}(\beta(x))$ for all $n, s \in \ZZ$ and
  $x \in M$.
\end{enumerate}
\end{thm}
\begin{proof}
  Let $\tilde{\beta} \in \Aut(M \rtimes_{\alpha}\ZZ)$. 
  Let $m, r \in M$ and $z_s \in \U(M)$ for all $s \in \ZZ$ and $\alpha \colon \ZZ \rightarrow \Aut(M)$ be the monomorphism with the covariant relation $\alpha_{n}(x)=U_{n}xU_n^{\ast}$.
  Since $\tilde{\beta}(m U_{n})=\beta
  (m)z_{n}U_{n}$, we have
  \begin{align}
    \label{coclerln}\tilde{\beta}(mU_{n}rU_{s})=\tilde{\beta}(mu_{n})\tilde{\beta}(rU_{s}) \punkt
  \end{align}
  Consequently, by the Equation~\eqref{coclerln} one has
  \begin{align}\label{impeq}
    \tilde{\beta}(m \alpha_{n}(r)U_{n+s})
      & = \tilde{\beta}(mU_{n}rU_{s})\\
      & = \beta(m)z_{n} U_{n} \beta(r)z_{s}U_{s}\\
      & = \beta(m)z_{n} \alpha_{n}(\beta(r))U_{n}z_{s}U_{n}^{\ast}U_{n+s}\\
      & = \beta(m)z_{n}\alpha_{n}(\beta(r))\alpha_{n}(z_{s})U_{n+s}\,.
  \end{align}
  By the Definition of $\tilde{\beta}$ and combining the Equation~\eqref{impeq}, it follows that
  \begin{align}\label{auteqn}
    \beta(m)\beta(\alpha_{n}(r))z_{n+s}U_{n+s}=\beta(m)z_{n}\alpha_{n}(\beta(r)z_{s})U_{n+s} \,.
  \end{align}
  Plugging $r=\id, m=\id$ into the Equation~\eqref{auteqn}, we obtain
  \begin{align}
    \label{Co}
    z_{n+s}=z_{n}\alpha_{n}(z_{s}) \,.
  \end{align}
  Inserting $m=\id$ in the Equation~\eqref{auteqn}, we have
  \begin{align*}
      \beta(\alpha_{n}(r))z_{n+s}=z_{n} \alpha_{n}(\beta(r))\alpha_{n}(z_{s})
    \end{align*}
    and by using the Equation~\eqref{Co}, we have
    \begin{align*}
     \beta(\alpha_{n}(r))z_{n}=z_{n} \alpha_{n}(\beta(r))\punkt
  \end{align*}
  Conversely, let  $z_{n+s}=z_{n}\alpha_{n}(z_{s})$
  and
  $\beta(\alpha_{n}(x))z_{n}=z_{n} \alpha_{n}(\beta(x))$ for all $n, s \in \ZZ$ and
  $x \in M$.
Then, we calculate that
\begin{align*}
  \tilde{\beta}\left(\sum_{n, m \in \ZZ}a_{n} U_{n} b_{m} U_{m}\right)
   &=\tilde{\beta}\left(\sum_{n, m \in \mathbb{Z}}a_{n} \alpha_{n}(b_{m})U_{n+m}\right) \\
   & =\sum_{n, m \in \ZZ}\beta(a_{n})\beta(\alpha_{n}(b_{m}))z_{n+m}U_{n+m}\\
   & =\sum_{n, m \in \ZZ}\beta(a_{n})\beta(\alpha_{n}(b_{m}))z_{n} \alpha_{n}(z_{m})U_{n+m}\\
   & =\sum_{n, m \in \ZZ}\beta(a_{n})z_{n}\alpha_{n}(\beta(b_{m})\alpha_{n}(z_{m})U_{n+m}\,.
\end{align*}
Note that
\begin{align*}
  \tilde{\beta}\left(\sum_{n \in \ZZ}a_{n} U_{n}\right)\tilde{\beta}\left(\sum_{m \in \mathbb{Z}}b_{m} U_{m}\right)
   & =\left(\sum_{n \in \ZZ}\beta(a_{n})z_{n} U_{n}\right) 
      \left(\sum_{m \in \mathbb{Z}}\beta(b_{m}) z_{m} U_{m}i\right) \\
   & =\sum_{n, m \in \ZZ}\beta(a_{n})z_{n} U_{n} \beta(b_{m})z_{m} U_{m}                          \\
   & =\sum_{n, m \in \ZZ}\beta(a_{n})z_{n} \alpha_{n}(\beta(b_{m}))U_{n}z_{m} U_{n}^{\ast}U_{n+m}    \\
   & =\sum_{n , m \in \ZZ}\beta(a_{n})z_{n} \alpha_{n}(\beta(b_{m}))\alpha_{n}(z_{m})U_{n+m}\,.
\end{align*}
Therefore $\tilde{\beta} \in \Aut(M\rtimes_{\alpha}\ZZ)$.
\end{proof}
Now we discuss necessary and sufficient conditions to extend a given automorphism from a factor to the twisted crossed product.
\subsection{Extension of automorphism to twisted crossed product}
\label{ssec:Extensionofautomorphismtotwistedcrossedproduct} Let $M \subseteq
\mathcal{B}(L^{2}((M, \varphi))$ be a factor and in standard form equipped with a faithful normal state $\varphi$.
Let $\gamma \colon \ZZ \times \ZZ \longrightarrow \U(1)$ be a 2-cocycle and $\alpha \colon \ZZ \rightarrow \Aut(M)$ be the monomorphism.
 We will consider the crossed product twisted by $\gamma$ given by
\begin{align*}
 M \rtimes_{\alpha, \gamma}\ZZ=(x, U_{n})
 \subseteq \mathcal{B}(L^{2}(\ZZ, L^{2}(M, \varphi)
 ) && x\in M, ~n \in \ZZ\,.
\end{align*}
We have the following cocycle relations:
\begin{align*}
  & \alpha_{n} \circ \alpha_{m}= \Ad \gamma(n, m) \circ \alpha_{n+m}&& \alpha_{n} \in \Aut(M), n, m \in \ZZ\\
  & \gamma(n, m) \gamma(n+m, \ell)=\alpha_{n}(\gamma(m, \ell)) \gamma(n, m+\ell)&& \ell \in \ZZ\\
  & U_{n} U_{m}=\gamma(n, m) U_{n+m}\\
  & U_{n}xU_{n}^{\ast}=\alpha_{n}(x)\,.
\end{align*}
We write any generic element of the twisted crossed product as $a:=\sum_{n \in \ZZ}a_{n} U_{n},
a_{n} \in M$.
Recall that $\Omega_{\varphi}$ is a cyclic and separating vector for $M$. 
Note that
$\Omega_{\varphi}\otimes \delta_{0}$ is a cyclic and separating vector for $M \rtimes_{\alpha, \gamma}\ZZ$.
The associated state $\tilde{\varphi}$ corresponding to $\varphi$ on $M \rtimes_{\alpha, \gamma}\ZZ$ is given by 
\begin{align}
    \tilde{\varphi}\left(\sum_{n \in \ZZ} a_{n} U_{n}\right)&=\varphi(a_{0})\,.
\end{align}
For $\beta\in \Aut(M)$, the map
$\hat{\beta} \colon M \rtimes_{\alpha, \gamma}\ZZ\rightarrow M \rtimes_{\alpha, \gamma}
\ZZ$
is given by
\begin{align}\label{DEFm}
  \hat{\beta}\left(\sum_{n \in \ZZ} a_{n} U_{n})\right)&=\sum_{n \in \ZZ} \beta(a_{n})z_{n} U_{n}\,.
\end{align}
Now we are going to prove similar result as Lemma \ref{WED}.\\

\begin{lem}\label{WEL}
  The map $\hat{\beta}\colon M \rtimes_{\alpha, \gamma}\ZZ \longrightarrow M \rtimes_{\alpha, \gamma}\ZZ$ is well-defined map.
\end{lem}
\begin{proof}
  Let $x=\sum x_{n} U_{n} =0$, where $x_n \in M$.
  We claim that $x_{n}=0$.
   Note that
  \begin{align*}
    0=\|x\|_{2, \tilde{\varphi}}^{2} & =\tilde{\varphi}(x^{\ast}x)\\
                                     & =\tilde{\varphi}\left(\sum_{n, m \in \mathbb{Z}}U_{n}^{*}x_{n}^{\ast}x_{m}U_{m}\right)\\
                                     & = \tilde{\varphi}\left(\sum_{n, m \in \ZZ}\gamma(-n, n)^{*}U_{-n}x_{n}^{\ast}U_{-n}^{\ast}U_{-n}x_{m} U_{-n}^{\ast}U_{-n}U_{m}\right) &(\text{by }U_{n}^{\ast}=\gamma(-n, n)^{\ast}U_{-n}) \\
                                     & =\tilde{\varphi}\left(\sum_{n, m \in \ZZ}\gamma(-n, n)^{\ast}\alpha_{-n}(x_{n}^{\ast})\alpha_{-n}(x_{m})\gamma(-n, m) U_{-n+m}\right)\\
                                     & =\varphi\left(\sum_{n \in \ZZ}\gamma(-n, n)^{\ast}\alpha_{-n}(x_{n}^{\ast}x_{n})\gamma(-n, n)\right)\,.
  \end{align*}
  As $\varphi$  is a normal state on $M$, we get                                                            
  \begin{align}
    \label{Stateeq}
     0=\sum_{n}\langle \gamma(-n, n)^{\ast}\alpha_{-n}(x_{n}^{\ast}x_{n})\gamma(-n, n) \Omega_{\varphi}, \Omega_{\varphi}\rangle\punkt
  \end{align}
  Observe that
  $\gamma(-n, n)^{\ast}\alpha_{-n}(x_{n}^{\ast}x_{n})\gamma(-n, n)$ is in particular positive. 
  Write
  \begin{align}\label{INTEQ}
    \gamma(-n, n)^{\ast}\alpha_{-n}(x_{n}^{\ast}x_{n})\gamma(-n, n)=y_{n}\punkt
  \end{align}
  Using Equation~\eqref{INTEQ} and inserting $y_n$ in the Equation~\eqref{Stateeq}, we have
  \begin{align*}
    0=\sum_{n}\|y_{n}^{1/2}\Omega_{\varphi}\|^{2}
  \end{align*}
  and therefore
  \begin{align*}
  y_{n}^{1/2}\Omega_{\varphi}=0\punkt
  \end{align*}
  Thus $y_{n}^{1/2}z'\Omega_{\varphi}=0$ where $y_{n}\in M$ and $z \in M' \cap \mathcal{B}(L^{2}(M, \varphi))$. As $\Omega_{\varphi}$ is a separating vector for $M$ and therefore cyclic for $M' \cap \mathcal{B}(L^{2}(M, \varphi))$,  we have
  $y_{n}^{1/2}=0$. 
  Hence $y_{n}^{\frac{1}{2}}=0$ implies $y_{n}=0$ because $y_{n}$ is positive. 
Therefore $\alpha_{-n}(x_{n}^{\ast}x_{n})=0$, which implies
  $x_{n}=0$. 
  Hence $\hat{\beta}$ is well-defined.
\end{proof}
\begin{thm}\label{extcptw}
Let $\hat{\beta}$ be the map defined in the Equation~\eqref{DEFm}. 
Then the following conditions are equivalent:
\begin{enumerate}
\item $\hat{\beta}\in \Aut(M \rtimes_{\alpha, \gamma}\ZZ)$
\item $\beta(\gamma(n, s))z_{n+s}=
  z_{n}\alpha_{n}(z_{s})\gamma(n, s)$ and
  $\beta(\alpha_{n}(x))z_{n}=z_{n} \alpha_{n}(\beta(x)) \text{ for all } n, s \in \ZZ, x \in M$.
\end{enumerate}
\end{thm}
\begin{proof}
  Firstly, assume that $\hat{\beta}\in \Aut(M \rtimes_{\alpha, \gamma}\ZZ)$.
  Let $m, r \in M$.
   Note that
  $\hat{\beta}(mU_{n})=\beta(m)z_{n}U_{n}$. 
  Thus we get
  \begin{align*}
    \hat{\beta}(mU_{n} r U_{s}) & =\hat{\beta}(mU_{n})\hat{\beta}(rU_{s})                                             \\
                        & =\beta(m)z_{n} U_{n} \beta(r)z_{s} U_{s}                                \\
                          & =\beta(m)z_{n}\alpha_{n}(\beta(r))\alpha_{n}(z_{s})\gamma(n, s) U_{n+s}
  \end{align*}
  and this gives
  \begin{align*}
     \hat{\beta}(m \alpha_{n}(r) \gamma(n, s) U_{n+s})=\beta(m)z_{n}\alpha_{n}(\beta(r))\alpha_{n}(z_{s})\gamma(n, s) U_{n+s}\,.
     \end{align*}
     Consequently, we get
     \begin{align}\label{INTEQ1}
      \beta(m)\beta(\alpha_{n}(r))\beta(\gamma(n, s))z_{n+s}U_{n+s}=\beta(m)z_{n}\alpha_{n}(\beta(r))\alpha_{n}(z_{s})\gamma(n, s) U_{n+s}\punkt
  \end{align}
  Inserting $r=\id, m=\id$ in the Equation~\eqref{INTEQ1}, we obtain
  \begin{align}
    \label{twistedcocyle}
    \beta(\gamma(n, s))z_{n+s}=z_{n}\alpha_{n}(z_{s})\gamma(n, s)\,.
  \end{align}
  Again inserting $m=\id$ in the Equation~\eqref{INTEQ1} and using the  Equation~\eqref{twistedcocyle}, we get
  \begin{align*}
    \beta(\alpha_{n}(r))z_{n}\alpha_{n}(z_{s})\gamma(n, s)=z_{n}\alpha_{n}(\beta(r))\alpha_{n}(z_{s})\gamma(n, s)
  \end{align*}
  and consequently
  \begin{align}
    \beta(\alpha_{n}(r))z_{n}=z_{n}\alpha_{n}(\beta(r)).
  \end{align}
  Conversely, we have
  \begin{align*}
     \hat{\beta}\left(\sum_{n \in \ZZ}a_{n} U_{n}\sum_{m \in \ZZ}b_{m}U_{m}\right)
     & =\hat{\beta}\left(\sum_{n, m \in \ZZ}a_{n}U_{n}b_{m}U_{m}i\right)\\
     & =\hat{\beta}\left(\sum_{n, m \in \ZZ}a_{n}\alpha_{n}(b_{m})\gamma(n, m)U_{n+m}\right)\\
     & =\sum_{n, m \in \ZZ}\beta(a_{n}) \beta(\alpha_{n}(b_{m}))\beta(\gamma(n, m))z_{n+m}U_{n+m}\\
     & =\sum_{n \in \ZZ}\beta(a_{n})\beta(\alpha_{n}(b_{m}))z_{n} \alpha_{n, m}(z_{m})\gamma(n, m) U_{n+m}\\
     & =\sum_{n, m \in \ZZ}\beta(a_{n})z_{n} \alpha_{n}(\beta(b_{m}))\alpha_{n}(z_{m})\gamma( n, m)U_{n+m}\,.
  \end{align*}
  Moreover, we get
  \begin{align*}
     \hat{\beta}\left(\sum_{n \in \ZZ}a_{n} U_{n}\right) \hat{\beta}\left(\sum_{m \in \ZZ}b_{m} U_{m}\right)
     & =\left(\sum_{n \in \ZZ}\beta(a_{n})z_{n}U_{n}\right)\left(\sum_{m}\beta(b_{m})z_{m} U_{m}\right)\\
     & =\sum_{n, m \in \ZZ}\beta(a_{n})z_{n} U_{n}\beta(b_{m})z_{m}U_{m}\\
     & =\sum_{n, m \in \ZZ}\beta(a_{n})z_{n} \alpha_{n}(\beta(b_{m}))\alpha_{n}(z_{m})\gamma(n, m) U_{n+m}\,.
  \end{align*}
  Hence, it follows that $\hat{\beta}\in \Aut(M \rtimes_{\alpha, \gamma}\ZZ)$,
  which concludes the proof.
\end{proof}

\end{appendix}

\def\cprime{$'$}\newcommand{\noopsort}[1]{}

\address

\begin{bibdiv}
\begin{biblist}

\bib{BiDeGi2023}{article}{
      author={Bischoff, Marcel},
      author={Del~Vecchio, Simone},
      author={Giorgetti, Luca},
       title={Quantum operations on conformal nets},
        date={2023},
        ISSN={0129-055X,1793-6659},
     journal={Rev. Math. Phys.},
      volume={35},
      number={4},
       pages={Paper No. 2350007, 30},
         url={https://doi.org/10.1142/S0129055X23500071},
      review={\MR{4589335}},
}

\bib{BrGuLo1993}{article}{
      author={Brunetti, R.},
      author={Guido, D.},
      author={Longo, R.},
       title={Modular structure and duality in conformal quantum field theory},
        date={1993},
        ISSN={0010-3616,1432-0916},
     journal={Comm. Math. Phys.},
      volume={156},
      number={1},
       pages={201\ndash 219},
         url={http://projecteuclid.org/euclid.cmp/1104253522},
      review={\MR{1234110}},
}

\bib{Bi2011}{article}{
      author={Bischoff, Marcel},
       title={{Models in Boundary Quantum Field Theory Associated with Lattices
  and Loop Group Models}},
        date={2011},
     journal={Arxiv preprint arXiv:1108.4889},
}

\bib{Bo1974}{book}{
      author={Bourbaki, N.},
       title={\'{E}l\'{e}ments de math\'{e}matique. {T}opologie
  g\'{e}n\'{e}rale. {C}hapitres 5 \`a 10},
   publisher={Hermann, Paris},
        date={1974},
        ISBN={2-7056-5692-8},
      review={\MR{3822133}},
}

\bib{ChGuLiWe2013}{article}{
      author={Chen, Xie},
      author={Gu, Zheng-Cheng},
      author={Liu, Zheng-Xin},
      author={Wen, Xiao-Gang},
       title={Symmetry protected topological orders and the group cohomology of
  their symmetry group},
        date={2013Apr},
     journal={Phys. Rev. B},
      volume={87},
       pages={155114},
         url={https://link.aps.org/doi/10.1103/PhysRevB.87.155114},
}

\bib{Co1973}{article}{
      author={Connes, Alain},
       title={{Une classification des facteurs de type {${\rm} III$}}},
        date={1973},
     journal={Ann. Sci. École Norm. Sup.(4)},
      volume={6},
       pages={133–252},
}

\bib{Co1977}{article}{
      author={Connes, Alain},
       title={Periodic automorphisms of the hyperfinite factor of type
  $\mathrm{II}_1$},
        date={1977},
        ISSN={0001-6969},
     journal={Acta Sci. Math. (Szeged)},
      volume={39},
      number={1-2},
       pages={39\ndash 66},
      review={\MR{0448101}},
}

\bib{DoLo1984}{article}{
      author={Doplicher, S.},
      author={Longo, R.},
       title={{Standard and split inclusions of von {N}eumann algebras}},
        date={1984},
        ISSN={0020-9910},
     journal={Invent. Math.},
      volume={75},
      number={3},
       pages={493–536},
         url={http://dx.doi.org/10.1007/BF01388641},
      review={\MR{735338 (86h:46095)}},
}

\bib{DR1975}{article}{
      author={Driessler, W.},
       title={{Comments on lightlike translations and applications in
  relativistic quantum field theory}},
        date={1975},
        ISSN={0010-3616},
     journal={Comm. Math. Phys.},
      volume={44},
      number={2},
       pages={133–141},
      review={\MR{0416336 (54 \#4411)}},
}

\bib{DeGi2018}{article}{
      author={Del~Vecchio, Simone},
      author={Giorgetti, Luca},
       title={Infinite index extensions of local nets and defects},
        date={2018},
        ISSN={0129-055X},
     journal={Rev. Math. Phys.},
      volume={30},
      number={2},
       pages={1850002, 58},
         url={https://doi.org/10.1142/S0129055X18500022},
      review={\MR{3757743}},
}

\bib{DoXu2006}{article}{
      author={Dong, Chongying},
      author={Xu, Feng},
       title={{Conformal nets associated with lattices and their orbifolds}},
        date={2006},
        ISSN={0001-8708},
     journal={Adv. Math.},
      volume={206},
      number={1},
       pages={279–306},
      eprint={math/0411499v2},
         url={http://dx.doi.org/10.1016/j.aim.2005.08.009},
}

\bib{EvGa2022}{article}{
      author={Evans, David~E.},
      author={Gannon, Terry},
       title={Reconstruction and local extensions for twisted group doubles,
  and permutation orbifolds},
        date={2022},
        ISSN={0002-9947,1088-6850},
     journal={Trans. Amer. Math. Soc.},
      volume={375},
      number={4},
       pages={2789\ndash 2826},
         url={https://doi.org/10.1090/tran/8575},
      review={\MR{4391734}},
}

\bib{FrReSc1989}{article}{
      author={Fredenhagen, K.},
      author={Rehren, K.-H.},
      author={Schroer, B.},
       title={{Superselection sectors with braid group statistics and exchange
  algebras. {I}.\ {G}eneral theory}},
        date={1989},
        ISSN={0010-3616},
     journal={Comm. Math. Phys.},
      volume={125},
      number={2},
       pages={201–226},
         url={http://projecteuclid.org/getRecord?id=euclid.cmp/1104179464},
      review={\MR{1016869 (91c:81047)}},
}

\bib{GaFr1993}{article}{
      author={Gabbiani, Fabrizio},
      author={Fröhlich, Jürg},
       title={{Operator algebras and conformal field theory}},
        date={1993},
        ISSN={0010-3616},
     journal={Comm. Math. Phys.},
      volume={155},
      number={3},
       pages={569–640},
}

\bib{GuLo1996}{article}{
      author={Guido, Daniele},
      author={Longo, Roberto},
       title={The conformal spin and statistics theorem},
        date={1996},
        ISSN={0010-3616},
     journal={Comm. Math. Phys.},
      volume={181},
      number={1},
       pages={11\ndash 35},
         url={http://projecteuclid.org/euclid.cmp/1104287623},
      review={\MR{1410566 (98c:81121)}},
}

\bib{Hen17}{incollection}{
      author={Henriques, Andr\'{e}},
       title={The classification of chiral {WZW} models by
  {$H^4_+(BG,\mathbb{Z})$}},
        date={2017},
   booktitle={Lie algebras, vertex operator algebras, and related topics},
      series={Contemp. Math.},
      volume={695},
   publisher={Amer. Math. Soc., Providence, RI},
       pages={99\ndash 121},
         url={https://doi-org.proxy.library.ohio.edu/10.1090/conm/695/13998},
      review={\MR{3709708}},
}

\bib{IzLoPo1998}{article}{
      author={Izumi, Masaki},
      author={Longo, Roberto},
      author={Popa, Sorin},
       title={{A {G}alois correspondence for compact groups of automorphisms of
  von {N}eumann algebras with a generalization to {K}ac algebras}},
        date={1998},
        ISSN={0022-1236},
     journal={J. Funct. Anal.},
      volume={155},
      number={1},
       pages={25–63},
         url={http://dx.doi.org/10.1006/jfan.1997.3228},
      review={\MR{1622812 (2000c:46117)}},
}

\bib{Jo1980}{article}{
      author={Jones, Vaughan F.~R.},
       title={Actions of finite groups on the hyperfinite type {${\rm
  II}_{1}$}\ factor},
        date={1980},
        ISSN={0065-9266},
     journal={Mem. Amer. Math. Soc.},
      volume={28},
      number={237},
       pages={v+70},
         url={http://dx.doi.org/10.1090/memo/0237},
      review={\MR{587749}},
}

\bib{Ka1998}{book}{
      author={Kac, V.G.},
       title={{Vertex algebras for beginners}},
   publisher={Amer Mathematical Society},
        date={1998},
        ISBN={082181396X},
}

\bib{Lo1979}{article}{
      author={Longo, Roberto},
       title={{Notes on algebraic invariants for noncommutative dynamical
  systems}},
        date={1979},
        ISSN={0010-3616},
     journal={Comm. Math. Phys.},
      volume={69},
      number={3},
       pages={195–207},
         url={http://projecteuclid.org/getRecord?id=euclid.cmp/1103905488},
      review={\MR{550019 (80j:46108)}},
}

\bib{LoRe1995}{article}{
      author={Longo, Roberto},
      author={Rehren, Karl-Henning},
       title={{Nets of Subfactors}},
        date={1995},
     journal={Rev. Math. Phys.},
      volume={7},
       pages={567–597},
      eprint={arXiv:hep-th/9411077},
}

\bib{Mg2005}{article}{
      author={Müger, Michael},
       title={{Conformal Orbifold Theories and Braided Crossed G-Categories}},
        date={2005},
        ISSN={0010-3616},
     journal={Comm. Math. Phys.},
      volume={260},
       pages={727–762},
         url={http://dx.doi.org/10.1007/s00220-005-1291-z},
}

\bib{Po1994-2}{article}{
      author={Popa, Sorin},
       title={Classification of amenable subfactors of type {II}},
        date={1994},
        ISSN={0001-5962},
     journal={Acta Math.},
      volume={172},
      number={2},
       pages={163\ndash 255},
         url={http://dx.doi.org/10.1007/BF02392646},
      review={\MR{1278111 (95f:46105)}},
}

\bib{Su1980}{article}{
      author={Sutherland, Colin~E.},
       title={Cohomology and extensions of von {N}eumann algebras. {I}, {II}},
        date={1980},
        ISSN={0034-5318},
     journal={Publ. Res. Inst. Math. Sci.},
      volume={16},
      number={1},
       pages={105\ndash 133, 135\ndash 174},
         url={https://doi.org/10.2977/prims/1195187501},
      review={\MR{574031}},
}

\bib{WaWo2015}{article}{
      author={Wagemann, Friedrich},
      author={Wockel, Christoph},
       title={A cocycle model for topological and {L}ie group cohomology},
        date={2015},
        ISSN={0002-9947},
     journal={Trans. Amer. Math. Soc.},
      volume={367},
      number={3},
       pages={1871\ndash 1909},
  url={https://doi-org.proxy.library.ohio.edu/10.1090/S0002-9947-2014-06107-2},
      review={\MR{3286502}},
}

\end{biblist}
\end{bibdiv}
\end{document}